\documentclass{article}
\usepackage[english]{babel}
\usepackage{csquotes}
\usepackage[letterpaper,top=2cm,bottom=2cm,left=3cm,right=3cm,marginparwidth=1.75cm]{geometry}
\usepackage{authblk}
\usepackage{amsmath, mathtools,enumerate}
\usepackage{amsthm}
\usepackage{amsfonts}
\usepackage{amssymb}
\usepackage{bbm}
\usepackage{stmaryrd}
\usepackage{graphicx}
\usepackage{titlesec}
\usepackage{mathrsfs}
\usepackage[colorlinks=true, allcolors=blue, breaklinks]{hyperref}
\usepackage[hyperpageref]{backref}
\usepackage{pdfpages}
\usepackage{comment}
\usepackage{dsfont}
\usepackage[normalem]{ulem}
\usepackage{caption, subcaption}
\numberwithin{equation}{section} 

\newcommand\numberthis{\addtocounter{equation}{1}\tag{\theequation}}

\theoremstyle{plain}
\newtheorem{theorem}{Theorem}[section]
\newtheorem{lemma}[theorem]{Lemma}

\newtheorem{corollary}[theorem]{Corollary}
\newtheorem{proposition}[theorem]{Proposition}

\newtheorem{remark}[theorem]{Remark}

\renewcommand{\P}{\mathbb{P}}

\newcommand{\E}{\mathbb{E}}

\newcommand{\N}{\mathbbm{N}}
\newcommand{\R}{\mathbb{R}}
\newcommand{\Z}{\mathbb{Z}}
\newcommand{\eps}{\varepsilon}

\newcommand{\Zperp}{Z^K_\perp}

\title{Mutant invasion in stochastic individual-based models with cyclic resident dynamics}
\author[1]{Sylvain Billiard}
\author[2]{Chloë Mian*}
\author[2,3]{Charline Smadi}
\author[4,5]{Andr\'as T\'obi\'as}

\affil[1]{Univ. Lille, CNRS, UMR 8198 – Evo-Eco-Paleo, F-59000 Lille, France}
\affil[2]{Univ. Grenoble Alpes, CNRS, Institut Fourier (UMR 5582), 38000 Grenoble, France}
\affil[3]{Univ. Grenoble Alpes, INRAE, LESSEM, 38000 Grenoble, France}
\affil[4]{Department of Computer Science and Information Theory, Budapest University of Technology and Economics, Műegyetem rkp. 3., H-1111 Budapest, Hungary}
\affil[5]{HUN-REN Alfréd
Rényi Institute of Mathematics, Reáltanoda utca 13–15, H-1053 Budapest, Hungary}
\date{}

\begin{document}

\maketitle

\vspace{-3em}
\begin{center}
*Corresponding author. E-mail: chloe.mian@univ-grenoble-alpes.fr
\end{center}

\begin{abstract}
When a resident community oscillates, can a rare mutant invade, and with which probability? We study a multitype birth-and-death process with interactions, in which population sizes are of the order of a carrying capacity $K$, and whose large-population limit, as $K\to\infty$, has a hyperbolically stable limit cycle for the $D-1$ resident types. We introduce a single individual of a new type. We prove that, as long as the mutant population is smaller than $\varepsilon K$ for a small $\varepsilon>0$, the resident population sizes divided by $K$ stay close to a moving point of the cycle, whose asymptotic phase moves by at most $O(\varepsilon)$ per period. We deduce that, for large $K$, the mutant dies out quickly if its growth rate averaged over one period, $f(\bar{\mathbf n})$, is negative, and that it invades with positive probability, in a time close to $\ln K/f(\bar{\mathbf n})$, if $f(\bar{\mathbf n})$ is positive. This justifies a criterion used heuristically in ecology for decades. The invasion probability, unlike $f(\bar{\mathbf n})$, depends on the phase of the cycle at which the mutant arises. We give it explicitly through Kendall's formula, and show that its average over the cycle, weighted by the birth rate of the mutant, equals $f(\bar{\mathbf n})$. In predator--prey, competitive and mutualistic examples, the invasion probability varies by a factor between $2$ and $600$ along the cycle, and individual-based simulations agree quantitatively with the theory. After invasion, the oscillations may vanish, persist on a cycle of higher dimension, or give way to chaotic coexistence.
\end{abstract}




\section{Introduction}

 \par Interactions among individuals are one of the primary mechanisms governing ecological dynamics across all levels of biological organization, from individual growth, reproduction, and mortality to the flow of matter and energy through ecosystems, which arises directly from the consumption of biomass. The general properties of these ecological dynamics have long been investigated (e.g.~through Lotka--Volterra systems), both to improve our understanding of biodiversity dynamics and to address applied questions, such as predicting the fate of endangered species or managing natural resources such as fish stocks.

\par The way interactions occur within and between species depends on individuals' traits, such as the ability of prey to avoid detection by predators through coloration \cite{hubbard2010}. Because these traits are represented as parameters in ecological models, they can profoundly influence population and community dynamics. Through the nonlinearities generated by species interactions, a wide range of dynamical outcomes may arise: species may i) exclude one another (e.g.~in competition \cite{tilman1982}), ii) coexist stably near an equilibrium (e.g.~in mutualistic interactions \cite{MBLS}), iii) exhibit periodic oscillations (e.g.~in predator--prey systems \cite{Turchin2003}), or iv) display chaotic fluctuations (e.g.~\cite{Smale1976,Vano,WX}). Treating traits as fixed parameters, that is, as intrinsic properties of populations or species, is nevertheless a strong simplification, because new variants continuously arise through immigration and mutation. This raises a central question: how are population and community dynamics affected by the evolution of traits? Conversely, how are traits involved in species interactions expected to evolve?

\par To address these questions, three main classes of models have been studied. First, a large body of work has examined the fate of a single mutant (or a rare variant) arising in a resident population assumed to be at equilibrium, thereby excluding periodic ecological dynamics (e.g.~in predator--prey systems \cite{loeuille2005}, mutualistic systems \cite{F02}, and competitive systems \cite{DieckmannDoebeli}). When analytical results are available, the study of these stochastic models often relies on comparisons with branching processes, in particular controlling that the resident population remains close to its equilibrium until the mutant either goes extinct or reaches a macroscopic size. This can be established using generalizations of classical results on the exit of stochastic processes from an attractive domain based on large-deviation theory \cite{FW84,dupuis1997weak,champagnat2007invasion,durrett2010}. Following successful mutant establishment, the subsequent population dynamics are typically approximated by deterministic dynamical systems. Although these models are grounded in stochastic processes, they do not address trait evolution when resident populations exhibit periodic oscillations.

\par A second class of models considers the evolution of species traits in the presence of periodic ecological dynamics using systems of ordinary differential equations (see \cite{Abrams2000} for a review). For example, \cite{abrams1999,jones2007} studied a system of three interacting species, two prey and one predator, described by a three-dimensional ordinary differential equation (ODE) model. They numerically computed the temporal average, over one cycle, of the growth rate of one prey species assumed to be rare in the subsystem formed by the other prey species and the predator. They then concluded that the subsystem could be invaded when this temporal average growth rate was positive. 
This criterion is exactly the sign of the quantity that we shall call the invasion fitness in~\eqref{def_inv_fitness} below. 

It was formalised as the invasion exponent, or invasion fitness, of a rare type in a fluctuating environment by Metz et al.~\cite{MetzNisbetGeritz1992} and Ferri\`ere and Gatto~\cite{FerriereGatto1995}. One outcome of the present paper is thus a justification, from an individual-based stochastic model, of a criterion which has been used in the ecological literature for decades: we prove that the sign of the temporal average of the growth rate does govern the alternative between extinction and invasion of a single mutant, and we complement it with the invasion probability, a quantity which the deterministic approach cannot reach.
A third, closely related class of models considers two coupled ODE subsystems: one describing changes in species densities and the other describing the evolution of species traits within the framework of quantitative genetics (e.g.~\cite{cortez2020}). These models implicitly assume that mutants arise at macroscopic initial densities across the entire range of possible trait values. As a result, they also assume that ecological and evolutionary processes occur on similar timescales. Hence, these models typically do not properly describe the invasion of a rare variant in a large population.

\par Stochastic models of invasion in a periodic environment do exist, but in all of them the periodicity is imposed from outside. The survival probability of a mutant, or of a pathogen, whose birth and death rates are given periodic functions of time has been computed in population genetics \cite{UeckerHermisson2011} and in epidemiology \cite{CarmonaGandon2020}, and Esser and Kraut \cite{EsserKraut2025} recently derived effective growth rates and a trait substitution sequence for an individual-based model whose parameters vary periodically in time. To our knowledge, no stochastic model has investigated the evolution of traits involved in species interactions when the oscillations are generated by these interactions themselves, even though such dynamics are common in many ecological contexts, including predator--prey interactions \cite{Turchin2003}, competitive interactions \cite{Gilpin1975,SinervoLively1996}, and mutualistic interactions \cite{MBLS}. In that situation the environment seen by the mutant is not a given function of time: it is produced by the resident populations, which are themselves subject to demographic noise. The goal of this paper is to fill this gap by providing a framework to determine the invasion probability of a mutant arising in a resident population that exhibits periodic oscillations generated by intrinsic ecological processes, that is, by interactions among individuals rather than by environmental fluctuations.
 
\par The model presented in the current paper is related to several previous works. In a seminal paper, Champagnat~\cite{champagnat2006microscopic} identified the survival probability and the asymptotic invasion time of a mutant arising in a resident population with competitive interactions in which coexistence is impossible. During a successful invasion, the population size, rescaled by its carrying capacity $K$, moves from the vicinity of the stable single-type equilibrium of the resident to that of the mutant in $\mathcal{O}(\ln K)$ time. This result was later extended to situations in which multiple types can coexist stably \cite{champagnat2011polymorphic}. Billiard and Smadi~\cite{BS19,billiard2017interplay} studied a related stochastic process in which periodic dynamics can arise among three competing types, allowing a mutant to invade during a single cycle that lasts $\mathcal{O}(\ln K)$ time. On a different timescale, when mutations occur more frequently than in~\cite{champagnat2006microscopic} (the so-called power-law mutation regime), cyclic evolutionary behaviour on the $\mathcal{O}(\ln K)$ timescale was observed in a model with horizontal gene transfer \cite{champagnat2021stochastic} and in models with asymmetric competition and general mutation graphs \cite{CoquilleKrautSmadi}.

\par In the present paper, we consider a birth-and-death process with interactions involving $D$ types of individuals, among which the first $D-1$ are the initially resident types and the $D$-th type is the mutant. Each individual independently produces an offspring of its own type at a rate that depends on the sizes of the $D$ subpopulations and dies at another rate that also depends on these population sizes. The periodic dynamics of the resident population arises from interactions among the $D-1$ resident types, whose large-population limit is described by a system of ODE exhibiting a hyperbolically stable periodic orbit (limit cycle). The period of these oscillations is therefore of order $\mathcal O(1)$ in the large-population limit. To study the conditions under which a single mutant of type $D$ can invade, using an approximation by a branching process, one must control the fact that the resident populations remain close to their deterministic large-population limit. 
This constitutes a natural extension of~\cite{champagnat2006microscopic}, and the difficulty is of a new nature. When the resident population sits at a stable equilibrium, knowing that it stays in a small neighbourhood of that equilibrium is enough to know the environment seen by the mutant. This is no longer the case along a limit cycle $\bar{\mathbf n}$: 
staying close to the cycle does not tell where the resident process is along it, and demographic fluctuations can make it drift along the cycle. We therefore do not compare the resident process with $\bar{\mathbf n}(t)$ itself. What we prove instead is that at any time it is close to $\bar{\mathbf n}(t+\theta)$ for a phase $\theta$ which moves by at most a quantity of order $\varepsilon$ at each period.
The deterministic ingredient which makes this possible is the asymptotic phase associated with a stable hyperbolic limit cycle. It turns out that the mean growth rate of the mutant over one period is insensitive to this phase drift, which is why the invasion time and the extinction-invasion alternative can be controlled over the whole invasion time scale, whereas the invasion probability, which does depend on the phase, requires a comparison over a bounded number of periods only.

\par After describing our model and presenting our main results, we apply our method to several ecological contexts, including predator--prey, competitive, and mutualistic interactions. We address the following questions: What is the invasion probability of a single mutant arising in a resident population exhibiting periodic oscillations? What is the fate of the population following the successful invasion of a mutant: coexistence or replacement? When coexistence is maintained, are periodic oscillations preserved or lost? These questions have long been present in the ecological literature, as it remains unclear whether trait evolution generally dampens or promotes periodic oscillations \cite{Abrams2000}. The first question is answered by Theorem~\ref{th_proba_inv} and Proposition~\ref{prop_proba_inv}; the last two are explored numerically on the three examples of Section~\ref{sec_examples}, where the macroscopic behaviour after invasion is that of the $D$-dimensional system~\eqref{dynsyst}.

\section{Model}\label{sec_model}

We consider a discrete birth and death process in continuous time and assume that the birth and death rates of all individuals depend on their type and on the population state, that is the numbers of all types of individuals. Individuals may be of $D$ possible types and their dynamics depend on a positive parameter $K$ which corresponds to the system carrying capacity and is a measure of the total amount of resource or space available in the environment. In the sequel we will be interested in the population behaviour for large $K$.

Let for $i \in \{1,...,D\}$ and $t \in \R_+$, $N_i(t)$ be the number of type $i$ individuals in the population at time $t$, and denote by $(e_i, 1 \leq i \leq D)$ the canonical basis of $\R^D$.
Our individual based model $\mathbf{N}^K=(N^K_{1}(t),\ldots,N^K_{D}(t),t\geq 0)$ is a Markov process with the following transition rates: 
$$
	\mathbf{N}^K\mapsto\begin{cases}
	\mathbf{N}^K +e_i &\text{ at rate }N^K_i b_i(\mathbf{N}^K/K)\\
	\mathbf{N}^K -e_i &\text{ at rate }N^K_i d_i(\mathbf{N}^K/K)\\
	\end{cases} 
$$
where the $(b_i, d_i, 1\le i\le D)$ are non-negative functions of class $\mathcal C^1$ on $\mathbb R_+^D$. We also assume that $d_D(\bar{\mathbf n}(t),0)$ is not identically zero along the resident orbit $\bar{\mathbf n}$ introduced below.

We introduce the corresponding process rescaled by the carrying capacity $K$ as 
$$\mathbf{Z}^K=(Z_1^K(t),\ldots,Z_D^K(t),t\geq 0)=(N^K_1(t)/K,\ldots,N^K_D(t)/K,t\geq 0).$$
The law of $\mathbf{Z}^K$ with initial state $\mathbf{z}$ is denoted by $\P_\mathbf{z}^K$.

We know (see Theorem 2.1 in \cite[p.~456]{ethier2009markov}) that, as $K\to\infty$, given convergence in probability of the initial condition $\mathbf{Z}^K(0)$, the process $\mathbf{Z}^K$ converges in probability, uniformly on finite time intervals to the solution $n(t)=(n_1(t),\ldots,n_D(t))$ of the dynamical system
\begin{equation}  \label{dynsyst}
 \dot n_{i} (t)=n_{i}(t)\big(b_i(n(t))-d_i(n(t))\big), \qquad i=1,\ldots,D. 
 \end{equation}
We may find various examples of rates satisfying such assumptions. The most studied examples in this context are the Lotka-Volterra type competitive interactions where the birth rates are constants ($b_i(N_1,...,N_D)=b_i$) and the death rates are of the form $d_i(N_1,...,N_D)=d_i+ \sum_{j=1}^D C_{i,j}N_j$, with $d_i, C_{i,j}$ non negative constants, and $C_{i,i}$ positive constants (see for instance  \cite{champagnat2006microscopic,champagnat2007invasion,bovier2019crossing,champagnat2021stochastic}). Certain predator--prey models also satisfy these assumptions and exhibit hyperbolically stable periodic solutions in some cases, such as the classical Rosenzweig--MacArthur model~\cite{RosenzweigMacArthur} that we will recall in Section~\ref{subsec:predator-prey}. Moreover, the class of two-dimensional models exhibiting mutualistic and parasitic parameter regimes studied recently in~\cite{MBLS} (see Section~\ref{subsec:mutualism}) also fits into our framework, just as the class of two-dimensional competitive models with horizontal gene transfer investigated in~\cite[Section 4]{BCFMT18}, see also~\cite{BCFMT16}.

We will assume that the dynamics restricted to the first $D-1$ types has a stable periodic orbit contained in the interior of the positive orthant of $\R^{D-1}$ with period $T$. We will call this orbit 
$$\bar{\mathbf n}=(\bar{\mathbf n}(t),t\in [0,T])=(\bar n_i(t),i \in \{1,...,D-1\},t \in [0,T]).$$
Then we assume that a type $D$ individual arrives as a mutant or a migrant. Notice that the period is arbitrary, as we can obtain any other positive and finite period via a time change.
We aim at describing the fate of this mutant, and how it impacts the global population behaviour.

\section{Results}
\label{sec:results}

To study the evolution of interactions among populations that coexist periodically in the event of a mutant's arrival, we need to control the
trajectory of the resident populations while the mutant population is still small. The difficulty is that we cannot simply prove that the rescaled process
$$\Zperp:=(Z_1^K,\ldots,Z_{D-1}^K)$$
stays close to the periodic orbit as a parametrized trajectory: over a time of order $\ln K$, the accumulated demographic fluctuations make it drift along the cycle. What we prove instead, in \eqref{small_bias}--\eqref{small_bias_phase}, is that it stays close to a moving point of the orbit, whose phase we are able to follow. We then characterize the probability of mutant invasion, that is, the probability that the mutant population reaches a size of the order of $K$ and can interact macroscopically with the resident populations (see Theorem \ref{th_proba_inv}). A first question is to know whether this probability is positive or null, a second one consists in looking for a quantitative expression of this probability when it is positive. Notice that the invasion probability depends on the resident population state when the mutant arrives. 

We prove in Theorem~\ref{th_proba_inv} that, up to an error which
vanishes as $\varepsilon\to0$, this probability is the survival probability
of a Galton--Watson branching process with mean offspring number
$\exp(Tf(\bar{\mathbf n}))$, whose offspring distribution depends on the
phase of the cycle at which the mutant arises, where $f(\bar{\mathbf n})$ is defined in \eqref{def_inv_fitness} below. The sign of $f(\bar{\mathbf n})$ indicates whether this probability is positive or zero; the probability itself is not determined by the mean offspring number alone, but we obtain an explicit expression for it in Proposition \ref{prop_proba_inv}, which we will compare with individual-based simulations in the examples considered in this paper.\\

For $\mathbf x=(x_1,\ldots,x_{D-1}) \in \R^{D-1}$ let us define $\mathbf r(\mathbf x) = (r_1(\mathbf x),\ldots,r_{D-1}(\mathbf x))$ via \begin{equation} \label{def_r} r_i(\mathbf x) = x_i(b_i(\mathbf x,0)-d_i(\mathbf x,0)) \end{equation}
that is, $\mathbf r(\mathbf x)$ is the vector of the first $D-1$ coordinates of the right-hand side of the system~\eqref{dynsyst} of ordinary differential equations (ODEs) at position $(x_1,\ldots,x_{D-1},0)$. In particular, if $t \mapsto (n_1(t),\ldots,$ $n_{D-1}(t),0)$ is a solution to the system, then $\mathbf{r}(n_1(t),\ldots,n_{D-1}(t))=(\dot n_1(t),\ldots,\dot n_{D-1}(t))$.

In the sequel, $\Vert . \Vert$ denotes the sup norm. Our general assumption will be that the periodic orbit under consideration is a stable and hyperbolic limit cycle, which means that the fixed point of the corresponding Poincar\'e map is stable and hyperbolic, that is, all eigenvalues $\mu$ of the $(D-2) \times (D-2)$ Jacobian matrix of the Poincar\'e map are located inside the unit circle $|\mu|<1$, see~\cite[Section 2.2.3]{kuznetsov} for details.

For $\alpha>0$ we introduce the $\alpha-$neighbourhood of $\bar{\mathbf n}$,
\begin{equation}\label{alpha-neighb}  \mathcal B_\alpha(\bar{\mathbf n})=\big\{ x \in \R^{D-1} \colon \exists t \in [0,T] \colon \Vert x - \bar{\mathbf n}(t) \Vert < \alpha \big\}. \end{equation}

The tool which allows us to follow the resident populations along the cycle is a second consequence of the stability and hyperbolicity of the limit cycle, and it will play a central role in the following. Not only is a trajectory starting close to $\bar{\mathbf n}$ attracted by the orbit, but it becomes indistinguishable from one particular solution running along it: there exists a map $\Theta$, defined on a neighbourhood of the orbit and with values in $\R/T\Z$, such that the solution started at $x$ is asymptotic to $t \mapsto \bar{\mathbf n}(t+\Theta(x))$. The quantity $\Theta(x)$ is called the asymptotic phase of $x$, and it satisfies $\Theta(\bar{\mathbf n}(\theta))=\theta$, so that it extends to a neighbourhood of the cycle the phase $\theta$ used in Section \ref{sec_examples}. We recall this classical statement, together with the properties of $\Theta$ that we shall need, in Lemma \ref{lem:phase}. In the sequel, for $x \in \mathcal B_\eps(\bar{\mathbf n})$, we denote by $\P_x$ the law of $\mathbf{Z}^K$ started at $\mathbf{Z}^K(0)=(x,1/K)$, that is when the $D-1$ resident populations start at $Kx$ and a single mutant is present.\\

For the sake of readability, we introduce the notations $\rho_D$ and $\sigma_D$ for the dynamics of the mutant for all $x \in \R_+^D$:
$$ \rho_D(x)=b_D(x)-d_D(x) \quad \text{and} \quad \sigma_D(x)= b_D(x)+d_D(x). $$

To state the next theorem, let us introduce the following notations, for an orbit $\bar{\mathbf n}$ of period $T$:
\begin{equation} \label{def_inv_fitness}
    f(\bar{\mathbf n}):= \frac{1}{T} \int_0^T \rho_D(\bar{\mathbf n}(s),0)ds.
\end{equation}
This quantity corresponds to the mean Malthusian parameter (that is individual growth rate) of a mutant if it arrives when the resident population sizes follow the periodic orbit $\bar{\mathbf n}$. Notice that, by periodicity, it does not depend on the phase of the orbit at which the mutant arises, unlike the invasion probability of a mutant, which does.

Besides, for $\theta \in \R/T\Z$, we introduce the continuous-time non homogeneous birth and death process $P_D^\theta:=(P_D^\theta(t), t \geq 0)$ with individual birth and death rates $b_D(\bar{\mathbf n}_\theta(t),0)$ and $d_D(\bar{\mathbf n}_\theta(t),0)$, where $\bar{\mathbf n}_\theta(t):=\bar{\mathbf n}((t+\theta) \mod T)$ denotes the resident orbit shifted by the phase $\theta$, as well as the discrete time Galton--Watson process $\mathscr{P}_D^\theta:=(\mathscr{P}_D^\theta(n)=P_D^\theta(nT), n \in \N)$. This is the process describing the mutant population as long as it is small, when it arises at the phase $\theta$ of the resident cycle. We denote by $p_D(\theta)$ their extinction probability
\begin{equation} \label{def_pD0}
 p_D(\theta):= \P_1(\exists k\in \N, \mathscr{P}_D^\theta(n)=0, \forall n \geq k)= \P_1(\exists t_0\geq 0, P_D^\theta(t)=0, \forall t \geq t_0).
\end{equation}

We finally introduce stopping times related to the process $N_D^K$:
\begin{equation} \label{defTK} \forall a \in \R_+, \quad T_a^K:= \inf \{t>0, N_D^K(t)= \lfloor aK \rfloor \}. \end{equation}
Then we have the following result.

\begin{theorem}\label{th_proba_inv}
Assume that \eqref{dynsyst} restricted to its $D-1$ first coordinates admits a stable and hyperbolic orbit $\bar{\mathbf n}$ of period $T$. Then there exist $\eps_0>0$ and finite positive constants $C_*$, $\tilde{C},$ $c_*$ and $C_{**}$ such that the following holds for every $\eps\leq\eps_0$.

$(i)$ For every finite $C$ and every $K$ large enough,
\begin{equation}\label{small_bias}
  \sup_{x\in\mathcal B_\eps(\bar{\mathbf n})}
  \P_x\Bigl(\exists\, t\le T_\eps^K\wedge C\ln K:\
     \Zperp(t)\notin\mathcal B_{C_*\eps}(\bar{\mathbf n})\Bigr)
   \leq \frac{8(D-1)}{T} C \ln K\;e^{-c_*K\eps^{2}} .
\end{equation}
More precisely, there exists an event of probability at least $1-(8(D-1)/T) C\ln K e^{-c_*K\eps^2}$ on which both the event of \eqref{small_bias} fails and, setting $\theta_n:=\Theta(\Zperp(nT))$, one has for every integer $n$ such that $(n+1)T\le T_\eps^K\wedge C\ln K$,
\begin{equation}\label{small_bias_phase}
  \sup_{t\in[nT,(n+1)T]}
  \bigl\Vert \Zperp(t)-\bar{\mathbf n}\bigl(t-nT+\theta_n\bigr)\bigr\Vert\le C_*\eps
  \qquad\text{and}\qquad
  |\theta_{n+1}-\theta_n|\le C_*\eps .
\end{equation}

$(ii)$ If $f(\bar{\mathbf n})<0$ then
\begin{equation}\label{extinction}
  \lim_{K\to\infty}\ \inf_{x\in\mathcal B_\eps(\bar{\mathbf n})}
   \P_x\bigl(T_0^K\le \ln\ln K\bigr)=1 .
\end{equation}

$(iii)$  If $f(\bar{\mathbf n})>0$ then
\begin{equation}\label{not_survive_ND}
  \limsup_{K\to\infty}\ \sup_{x\in\mathcal B_\eps(\bar{\mathbf n})}
  \Bigl|\P_x\bigl(T_0^K<T_\eps^K\bigr)-p_D\bigl(\Theta(x)\bigr)\Bigr|
  \ \le\ C_{**}\sqrt{\eps}
\end{equation}
and
\begin{equation}\label{survive_ND}
  \limsup_{K\to\infty}\ \sup_{x\in\mathcal B_\eps(\bar{\mathbf n})}
  \left|\P_x\Bigl(\frac{\ln K}{f(\bar{\mathbf n})+\tilde{C}\eps}
    <T_\eps^K<\frac{\ln K}{f(\bar{\mathbf n})-\tilde{C}\eps}\Bigr)
   -\bigl(1-p_D(\Theta(x))\bigr)\right|\ \le\ C_{**}\sqrt{\eps} .
\end{equation}
\end{theorem}

The critical case $f(\bar{\mathbf n})=0$ is not covered by Theorem~\ref{th_proba_inv}; by Proposition~\ref{prop_proba_inv} below, the limiting branching process then dies out almost surely.
In the limit when $T$ tends to $0$, the periodic orbit $(\bar{\mathbf n}(t),t\in [0,T])$ shrinks to an equilibrium $\bar{\mathbf n}(0)$, and in this case $f(\bar{\mathbf n})$ simply corresponds to $\rho_D(\bar{\mathbf n}(0))$.
In this limiting setup, the underlying branching process is a classical homogeneous binary Galton--Watson process, and in the supercritical case its survival probability is ${\rho_D(\bar{\mathbf n}(0))}/{b_D(\bar{\mathbf n}(0))}$ and its exponential growth rate conditional on survival equals $\rho_D(\bar{\mathbf n}(0))$ \cite[p.~202]{AN72}. The analogues of~\eqref{extinction}-\eqref{survive_ND} corresponding to this case were derived in~\cite{champagnat2006microscopic}.\\

From Theorem \ref{th_proba_inv} we see that the sizes of the populations $1$ to $D-1$ are little affected by the growth of the mutant population as long as the latter has a size smaller than $\eps K$. If this latter reaches size $\eps K$ it has a macroscopic impact on the system and the dynamics of the whole system is well approximated by the $D$-dimensional system \eqref{dynsyst}. We will explore the behaviour of this system, mostly using numerical simulations, for the different examples under consideration (see Section \ref{sec_examples}). \\

Theorem \ref{th_proba_inv} states that the type $D$ population size dynamics is very close to that of the branching process $P_D^\theta$ as long as it is smaller than $\eps K$ for a small $\eps$. We thus recover a similar behaviour as in the classical analysis of a mutant invasion when the resident population is at an attracting stable equilibrium (see \cite{champagnat2006microscopic} for instance), the phase being the only additional ingredient. 
The study of the process $N^K_D$ is however more involved, because the environment seen by the mutant is the resident process itself, and staying close to the cycle does not tell where this process is along it. This is the content of \eqref{small_bias_phase}: the rescaled resident population sizes are not only close to the cycle as a set, but at any time they are close to a moving point $\bar{\mathbf n}(t+\theta)$ of the cycle, whose phase $\theta$ moves by at most $C_*\eps$ at each period. Over a time of order $\ln K$, these per-period estimates only bound the accumulated phase drift by $O(\eps\ln K)$, so that we do not know at a large time $t$ whether $\Zperp(t)$ is close to $\bar{\mathbf n}(t \pmod T)$. What \eqref{small_bias_phase} provides is the weaker, but sufficient, information that the trajectory stays close to some point of the orbit, whose phase we are able to follow.\\

The reason why this phase drift is harmless for most of our statements is the following elementary observation: for any continuous $g$ and any $\theta$, we have 
$$\int_0^T g(\bar{\mathbf n}(s+\theta))ds= \int_0^T g(\bar{\mathbf n}(s))ds,$$ 
so that the mean Malthusian parameter of the mutant over one period does not depend on the phase at which this period begins:
\begin{equation}\label{eq:phase-inv}
  \frac1T\int_0^T\rho_D(\bar{\mathbf n}_\theta(s),0)\,ds=f(\bar{\mathbf n})
  \qquad\text{for every }\theta .
  \end{equation}
  The alternative between \eqref{extinction} and \eqref{not_survive_ND}, as well as the invasion time in \eqref{survive_ND}, are
therefore insensitive to the moment of the cycle at which the mutant arises. The probability of invasion, on the contrary, does depend on it through $p_D(\Theta(x))$, and the examples of Section \ref{sec_examples} show that this dependence is quantitatively strong: along the resident cycle it varies by a factor $2$ in the predator--prey model, by a factor $6$ in
the mutualistic one, and by a factor of about $600$ in the competitive one (from $0.0013$ to $0.78$). Notice that if the resident populations start exactly on the orbit, $x=\bar{\mathbf n}(\theta)$, then $\Theta(x)=\theta$ and $p_D(\Theta(x))$ is exactly the quantity computed in Section \ref{sec_examples}.\\

Unlike in the equilibrium case, the value and sign of the instantaneous Malthusian growth rate $\rho_D(\bar{\mathbf n}_\theta(t),0)$ change throughout a period $T$, so that the behaviour of the process $P_D^\theta$ depends on the phase $\theta$. Its extinction probability remains nevertheless explicit, because $P_D^\theta$ is a linear birth and death process in a $T$-periodic environment, to which Kendall's formula \cite{Kendall1948} applies.

\begin{proposition}\label{prop_proba_inv}
Recall the definitions of $f(\bar{\mathbf n})$ and $p_D(\theta)$ in \eqref{def_inv_fitness} and \eqref{def_pD0}, respectively. Then
\begin{enumerate}
    \item If $f(\bar{\mathbf n})\le 0$, then $p_D(\theta)=1$ for every $\theta$.
    \item If $f(\bar{\mathbf n})>0$, then $p_D(\theta)<1$ for every $\theta$ and
\begin{align}\label{eq:kendall-sec3}
  1-p_D(\theta)
 & =\frac{1-e^{-Tf(\bar{\mathbf n})}}
        {\displaystyle\int_0^{T}b_D\bigl(\bar{\mathbf n}_\theta(t),0\bigr)\,
          e^{-\int_0^{t}\rho_D(\bar{\mathbf n}_\theta(u),0)du}\,dt} \nonumber \\
 & =\frac{1-e^{-Tf(\bar{\mathbf n})}}
        {1-e^{-Tf(\bar{\mathbf n})}+\displaystyle\int_0^{T}d_D\bigl(\bar{\mathbf n}_\theta(t),0\bigr)\,
          e^{-\int_0^{t}\rho_D(\bar{\mathbf n}_\theta(u),0)du}\,dt}\, .
\end{align}
\end{enumerate}
\end{proposition}

The two expressions in \eqref{eq:kendall-sec3} coincide because $\int_0^T \rho_D(\bar{\mathbf n}_\theta(t),0)e^{-\int_0^t \rho_D(\bar{\mathbf n}_\theta(u),0)du}dt = 1-e^{-Tf(\bar{\mathbf n})}$, and the second one is the more convenient for comparison purposes. When $b_D$ and $d_D$ are constant, both reduce to $(b_D-d_D)/b_D$, and we recover the survival probability of a homogeneous Galton--Watson process. In general, whether periodic oscillations make the invasion of a mutant easier or harder than at a fixed point can be read directly on \eqref{eq:kendall-sec3}, and depends on the way the birth rate $b_D$ along the orbit is correlated with the weight $e^{-\int_0^t \rho_D(\bar{\mathbf n}_\theta(u),0)du}$: no general conclusion holds, and the answer depends on the ecological context. Corollary~\ref{cor:average} below makes this statement quantitative. We shall compute both the phase-dependent probability \eqref{eq:kendall-sec3} and its average over the cycle in the examples of Section \ref{sec_examples}, and compare them with individual-based simulations. Let us finally emphasize that the sign of $f(\bar{\mathbf n})$ is precisely the criterion used numerically by \cite{abrams1999,Abrams2000,jones2007} to decide whether a rare species can invade a periodically oscillating community, and that Theorem~\ref{th_proba_inv} shows it to be the right one for the underlying stochastic model. Proposition~\ref{prop_proba_inv} adds to it the invasion probability, which has no deterministic counterpart: in a deterministic model, a rare type with positive average growth rate always invades.

\begin{corollary}\label{cor:average}
Assume $f(\bar{\mathbf n})>0$, and write $b(\theta):=b_D(\bar{\mathbf n}(\theta),0)$, $\rho(\theta):=\rho_D(\bar{\mathbf n}(\theta),0)$ and $s(\theta):=1-p_D(\theta)$ for the birth rate, the growth rate and the survival probability of a mutant arising at phase $\theta$. Let
\[
  \bar s:=\frac1T\int_0^Ts(\theta)\,d\theta,\qquad
  \bar b:=\frac1T\int_0^Tb(\theta)\,d\theta,\qquad
  \operatorname{Cov}(b,s):=\frac1T\int_0^Tb(\theta)s(\theta)\,d\theta-\bar b\,\bar s
\]
be the averages of $s$ and $b$ over one period and their covariance over one period. Then $s$ is the unique positive $T$-periodic solution of
\begin{equation}\label{eq:riccati}
  s'(\theta)=b(\theta)s(\theta)^2-\rho(\theta)s(\theta),
\end{equation}
and
\begin{equation}\label{eq:average}
  \frac1T\int_0^T b(\theta)\,s(\theta)\,d\theta=f(\bar{\mathbf n}).
\end{equation}
In particular, since $\bar b\ge f(\bar{\mathbf n})>0$,
\[
  \bar s=\frac{f(\bar{\mathbf n})-\operatorname{Cov}(b,s)}{\bar b},
\]
and if $b_D$ is constant along the cycle, this reduces to $\bar s=f(\bar{\mathbf n})/b_D$, as for a homogeneous birth and death process.
\end{corollary}

Theorem~\ref{th_proba_inv}, Proposition~\ref{prop_proba_inv} and  Corollary \ref{cor:average} will be proven in Section~\ref{sec_proof}.

\begin{remark}
Equation \eqref{eq:riccati} is the backward Kolmogorov equation for the survival probability of a lineage started at phase $\theta$, and
\eqref{eq:average} extends the homogeneous formula $s=\rho_D/b_D$: whatever the shape of the cycle, the birth-weighted average of the invasion
probability is the invasion fitness.
\end{remark}

\section{Examples}

\label{sec_examples}

We illustrate Theorem~\ref{th_proba_inv} with three examples for which the resident population converges to a stable hyperbolic cycle.
For each example, we first give the individual birth and death rates.
We then study the invading type when it is rare, using a linear birth-and-death branching process in a periodic environment.
Finally, we study the deterministic system after establishment.
The three examples give different outcomes: replacement and convergence to an equilibrium, coexistence on a new cycle, or coexistence in a numerical chaotic dynamics.

We use the notation of Sections~2 and~3: the resident cycle is denoted by $\bar n$ and its period by $T$.
The invasion fitness $f(\bar n)$ is defined in~\eqref{def_inv_fitness}.

For an introduction at phase $\theta\in[0,T)$, the survival probability
$s(\theta)=1-p_D(\theta)$ is given by \eqref{eq:kendall-sec3}, and its average over the cycle is denoted by $\bar s$, as in Corollary~\ref{cor:average}.
In the figures, phases are normalised by $T$, and we write
$q_K(\theta):=\mathbb P_{\bar{\mathbf n}(\theta)}(T^K_\varepsilon<T^K_0)$
for the corresponding finite-$K$ probability. In each example we specify
which point of the cycle is taken as phase $0$.

The finite-population simulations use the Gillespie algorithm~\cite{gillespie1977exact} to simulate the corresponding continuous-time birth--death process.
We define an invasion as the event that the invading population reaches $\lfloor \varepsilon K \rfloor$ before extinction.
For each of the three examples, we use 16 introduction phases and 1000 independent simulations per phase.
The error bars are 95\% Wilson confidence intervals.

\subsection{Example 1: Loss of periodic dynamics in a Rosenzweig--MacArthur predator--prey system}
\label{subsec:predator-prey}

Our first example is the Rosenzweig--MacArthur predator--prey model \cite{RosenzweigMacArthur}.
The resident prey and predator oscillate on a stable cycle.
We determine how the introduction phase affects the invasion of a mutant predator and then show that its establishment causes the loss of these periodic dynamics.

\subsubsection{Individual-based model and resident cycle}

We consider a prey type $x$, a resident predator $y$, and a mutant predator $z$.
Their per-capita birth and death rates are
\begin{equation}
	\begin{aligned}
		b_x(x,y,z) & =r,
		           & d_x(x,y,z)                              & =\frac{r}{N}x+\frac{c}{a+x}y
		+\frac{\widetilde c}{\widetilde a+x}z,                                              \\
		b_y(x,y,z) & =\frac{b x}{a+x},
		           & d_y(x,y,z)                              & =d,                          \\
		b_z(x,y,z) & =\frac{\widetilde b x}{\widetilde a+x},
		           & d_z(x,y,z)                              & =\widetilde d.
	\end{aligned}
	\label{eq:rma-rates}
\end{equation}
The two predators interact indirectly because they use the same prey.
Setting $z=0$ in the deterministic limit~\eqref{dynsyst} gives the Rosenzweig--MacArthur system for the residents:
\begin{equation}
	\left\{
	\begin{aligned}
		\dot x & =r x\left(1-\frac{x}{N}\right)-\frac{cxy}{a+x}, \\
		\dot y & =y\left(\frac{b x}{a+x}-d\right).
	\end{aligned}
	\right.
	\label{eq:rma-resident}
\end{equation}
For $b>d$, its positive equilibrium is
\[
	x^\star=\frac{ad}{b-d},
	\qquad
	y^\star=\frac{r}{c}\left(1-\frac{x^\star}{N}\right)(a+x^\star).
\]
It is unstable when $N>a+2x^\star$, in which case the positive quadrant contains an attracting periodic orbit~\cite{RosenzweigMacArthur}.
With the following parameter values,
\begin{equation}
(r,N,a,b,c,d)=\left(1.5,3.2,\frac{4}{3},1,2,0.4\right),
\label{eq:rma-resident-parameters}
\end{equation}
the corresponding equilibrium is $(x^\star,y^\star)=(0.8889,1.2037)$, and numerical integration gives a stable cycle with period $T\simeq12.5774$.
Numerically, the Floquet multipliers are $1$ and approximately $0.810$. The multiplier $1$ corresponds to motion along the cycle. 
The other multiplier measures perturbations away from the cycle, since $|0.810|<1$ these perturbations decrease after each period. 
Then, the resident cycle is hyperbolic and locally attracting, as required in Theorem~\ref{th_proba_inv}.

\subsubsection{Invasion of a mutant predator}

The parameter values for the mutant predator were chosen as follows:
\begin{equation}
	(\widetilde a,\widetilde b,\widetilde c,\widetilde d)
	=\left(\frac{20}{3},5,10,0.45\right).
	\label{eq:rma-mutant-parameters}
\end{equation}
When $z$ is rare, its effect on prey mortality and on the macroscopic dynamics of the resident pair $(x,y)$ is negligible.
Its population can then be approximated by a linear birth-and-death branching process in a periodic environment, with rates
\begin{equation}
	b_D\bigl(\bar n(t),0\bigr)
	=\frac{\widetilde b\,\bar x(t)}{\widetilde a+\bar x(t)},
	\qquad
	d_D\bigl(\bar n(t),0\bigr)=\widetilde d.
	\label{eq:rma-rare-rates}
\end{equation}
We then obtain
\begin{equation}
	f(\bar n)=\frac{1}{T}\int_0^T
	\left(\frac{\widetilde b\,\bar x(t)}{\widetilde a+\bar x(t)}
	-\widetilde d\right)\,\mathrm dt
	\simeq0.15286>0.
	\label{eq:rma-fitness}
\end{equation}
The survival probability varies between $0.173$ and $0.337$, and $\bar s\simeq0.249$.
It is maximal while the prey density is high and increasing, and the resident predator is close to its minimum.
Conversely, it is minimal shortly after the resident predator reaches its maximum, when the prey density is decreasing.
Figure~\ref{fig:rma-phase} compares this limit with $q_K(\theta)=\mathbb P_\theta(T_\varepsilon^K<T_0^K)$, estimated by Gillespie simulations for $K=10^4$ and $\varepsilon=0.01$.

\begin{figure}[t]
	\centering
	\includegraphics[width=0.96\textwidth]
	{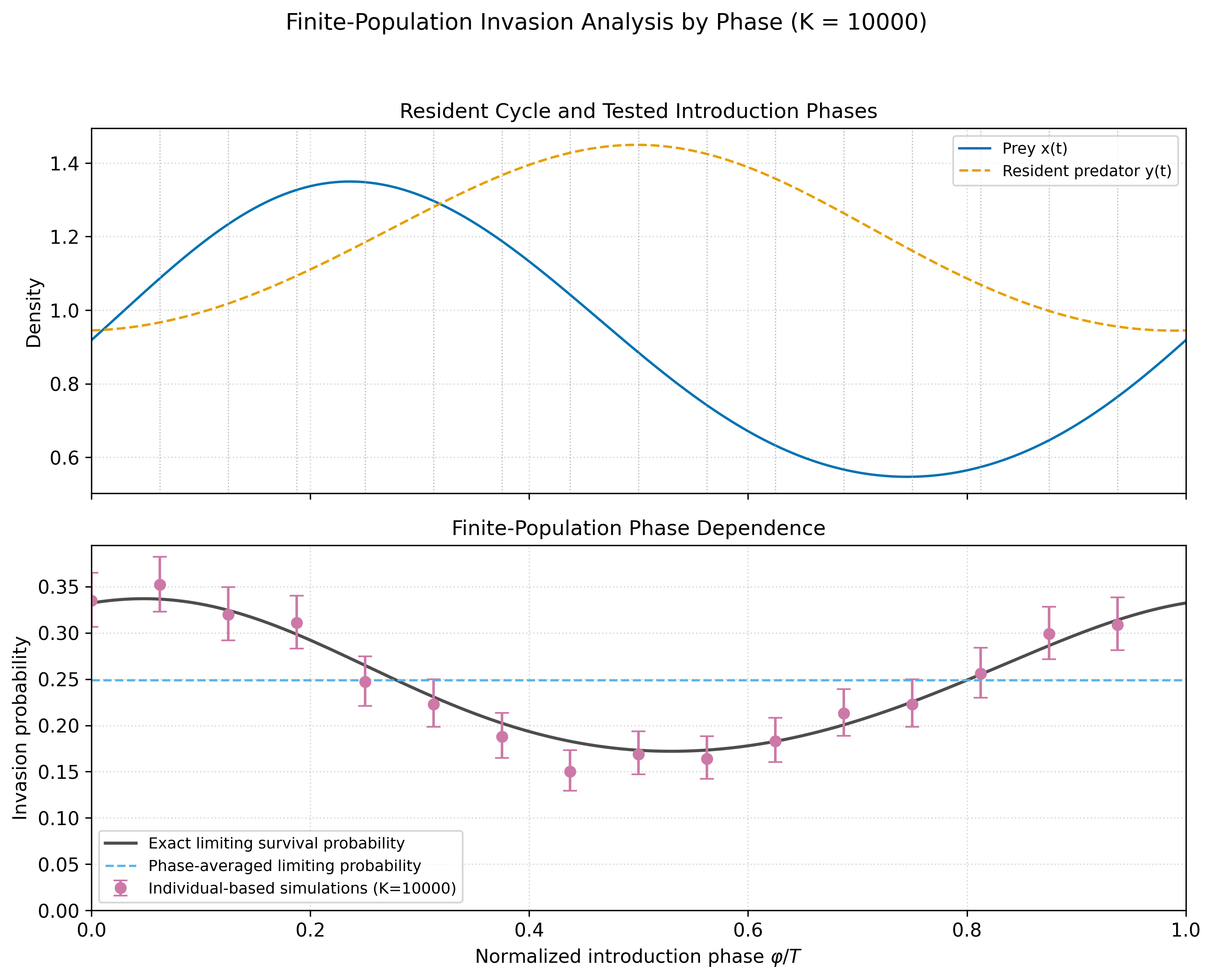}
	\caption{Invasion of a mutant in the Rosenzweig--MacArthur predator--prey system.
		The upper panel shows one period of the resident cycle and the introduction phases used in the simulations.
		The lower panel compares $s(\theta)$ with the individual-based estimate $q_K(\theta)$ for $K=10^4$.
	}
	\label{fig:rma-phase}
\end{figure}

\newpage
\subsubsection{Post-establishment dynamics}

Once the mutant density is macroscopic, the deterministic limit becomes
\begin{equation}
	\left\{
	\begin{aligned}
		\dot x & =r x\left(1-\frac{x}{N}\right)
		-\frac{cxy}{a+x}-\frac{\widetilde c xz}{\widetilde a+x}, \\
		\dot y & =y\left(\frac{b x}{a+x}-d\right),               \\
		\dot z & =z\left(\frac{\widetilde b x}{\widetilde a+x}
		-\widetilde d\right).
	\end{aligned}
	\right.
	\label{eq:rma-full-system}
\end{equation}
After establishment, numerical integration shows that $z$ replaces $y$.
The system converges to $(x,y,z)\simeq(0.6593,0,0.8725)$.
At this equilibrium, the per-capita growth rate of a rare resident predator is approximately $-0.0691$.
Thus, $y$ cannot reinvade locally. Hence, after invasion of the mutant predator and the extinction of the resident predator, the periodic dynamics are lost (Fig. \ref{fig:rma-post-invasion}).

\begin{figure}[t]
	\centering
	\includegraphics[width=0.96\textwidth]
	{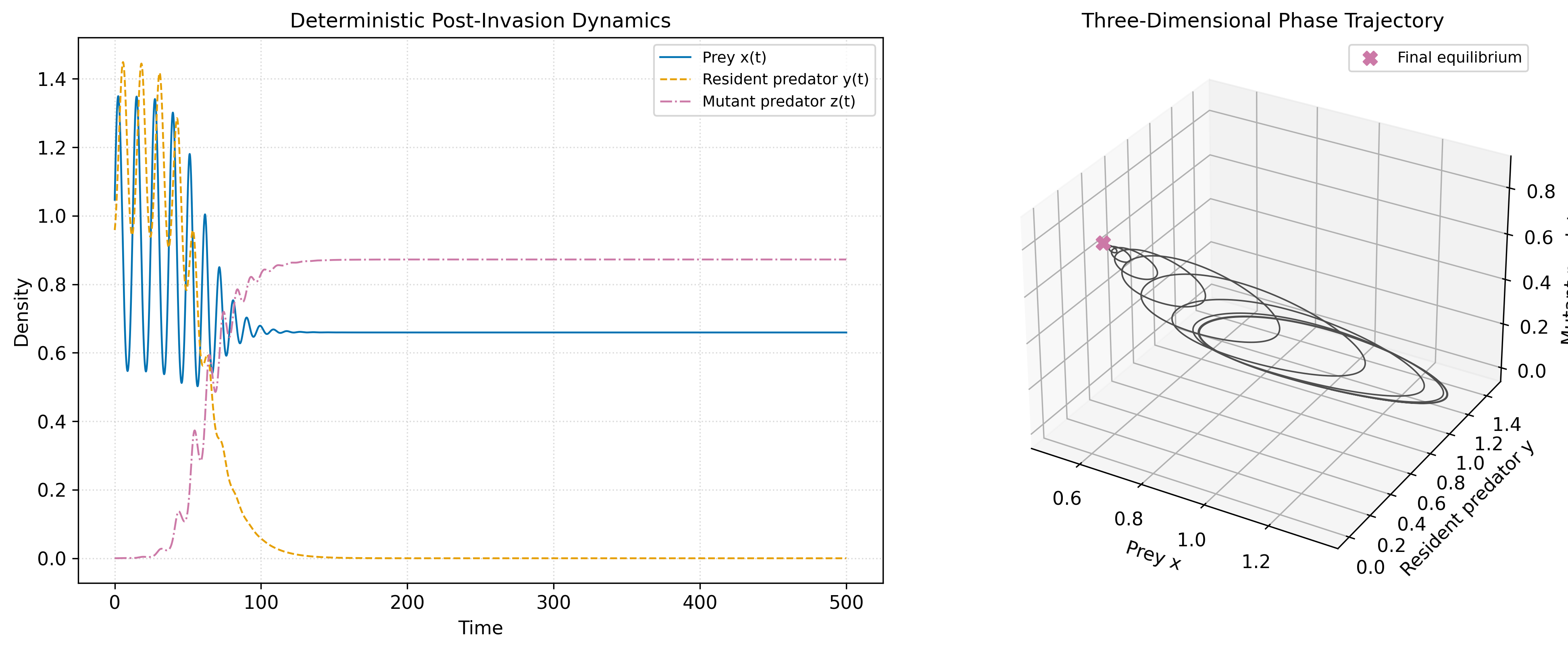}
	\caption{Predator--prey dynamics after establishment.
		The mutant replaces the resident predator, and the trajectory converges from the resident cycle to a stable prey--mutant equilibrium.
	}
	\label{fig:rma-post-invasion}
\end{figure}

\newpage
\subsection{Example 2: Emergence of chaotic dynamics in a competitive system}
\label{subsec:competition}

Our second example concerns competition, for which every interspecific effect decreases the per-capita growth rate of the affected species.
We start from three resident competitors on a stable cycle and introduce a fourth competing type.
The model is constructed from the four species system of Wang and Xiao \cite{WX}.
For some establishment phases, the post-invasion dynamics exhibit persistent coexistence on a numerically chaotic attractor.

\subsubsection{A three-species resident cycle}

Let $x=(x_1,x_2,x_3)^\mathsf T$ denote the resident densities, $B^{\rm cyc}_{ij}$ the effect of competition from a $j$ individual on the death rate of an $i$ individual,  and set $r=(0.42,0.57,0.72)^\mathsf T$.
At the individual level, a resident individual of type $i$ has the non-negative per-capita birth and death rates
\begin{equation}
	b_i(x)=r_i,
	\qquad
	d_i(x)=r_i\sum_{j=1}^3 B^{\rm cyc}_{ij}x_j,\qquad 1\leq i\leq3.
	\label{eq:competition-resident-rates}
\end{equation}

Consequently, in the absence of the invading type, the resident deterministic limit is the competitive Lotka--Volterra system

\begin{equation}
	\dot x_i=r_i x_i\left(1-\sum_{j=1}^3
	B^{\rm cyc}_{ij}x_j\right),\qquad 1\leq i\leq3.
	\label{eq:competition-resident}
\end{equation}
The following matrix of competition coefficients allows for an attracting limit cycle:
\begin{equation}
	B^{\rm cyc}=
	\begin{pmatrix}
		0.486518 & 0.026041 & 1.172464 \\
		0.545422 & 0.785093 & 0.105706 \\
		0.438029 & 1.126036 & 0.620741
	\end{pmatrix}.
	\label{eq:Bcycle}
\end{equation}

After discarding a transient phase in the simulations, we record successive upward crossings of the section $x_1=x_1^\star$, where $x^\star$ is the interior equilibrium. The crossing points converge to a fixed point of the associated Poincar\'e return map, and the time between two successive returns gives the period $T\simeq76.5396$. 
The eigenvalues of the linearized return map, equivalently the two non-trivial Floquet multipliers, are approximately $0.87097$ and smaller than $10^{-15}$ (a Floquet multiplier cannot vanish, since the monodromy matrix is invertible)~\cite[Chapter~2]{kuznetsov}. 
Since both have modulus strictly smaller than one, perturbations transverse to the orbit contract after one period. 
This numerically verifies that the resident cycle is hyperbolic and locally attracting, as required in Theorem~\ref{th_proba_inv}.

\subsubsection{Invasion of a new competitive species}

\label{subsection_invasion}

We now introduce a fourth competitor with density $m$. 
We call it a mutant throughout the example, although its interaction coefficients differ substantially from those of the three resident types.
It may equivalently be interpreted as a new competitor introduced by immigration. 
The theory does not require the mutational effect to be small: $m$ plays the role of type $D$ in Sections~\ref{sec_model} and~\ref{sec:results} and initially corresponds to one individual in a resident population of order $K$.

Write $u=(m,x_1,x_2,x_3)^\mathsf T$, set $r_m=1$, and recall that $r=(0.42,0.57,0.72)^\mathsf T$ for the residents.
To obtain a post-invasion regime different from the resident cycle, we use the following four-species interaction matrix, written in the coordinate order $(m,x_1,x_2,x_3)$:

\begin{equation}
	C=\begin{pmatrix}
		0.1     & 0.0588 & 2.156 & 0.49   \\
		0.03528 & 1      & 0.049 & 1.2936 \\
		0.13818 & 1.0486 & 1     & 0.1666 \\
		0.02548 & 0.5292 & 1.127 & 1
	\end{pmatrix}.
	\label{eq:competition-C}
\end{equation}
It is adapted from the four-species competitive system of Wang and Xiao at parameter value $s=0.98$~\cite{WX}, after the density rescaling of its first species $m=10u_1$.
Let $B^{\rm ch}$ be the lower-right $3\times3$ block of $C$ and define
\begin{equation}
	h(m)=\frac{m^8}{\delta^8+m^8},\qquad
	B(m)=B^{\rm cyc}+h(m)\bigl(B^{\rm ch}-B^{\rm cyc}\bigr),\quad
	\delta=0.01.
	\label{eq:competition-interpolation}
\end{equation}
The Hill function \color{blue} $h$ \color{black} gives a smooth transition between the two matrices.
When $m=0$, the resident matrix is $B^{\rm cyc}$. As $m$ increases, it approaches $B^{\rm ch}$. Such a dependence of the interaction between two species on the density of a third one is known as a higher-order interaction~\cite{BaireyKelsicKishony2016}. Here it lets the invading type modify the competition among the residents.
The individual-based model is specified by the following non-negative per-capita rates:
\begin{equation}
	\begin{aligned}
		b_m(u) & =r_m,
		       & d_m(u) & =r_m\left(C_{mm}m+\sum_{j=1}^3 C_{mj}x_j\right), \\
		b_i(u) & =r_i,
		       & d_i(u) & =r_i\left(C_{im}m+
		\sum_{j=1}^3B_{ij}(m)x_j\right),\quad 1\leq i\leq3.
	\end{aligned}
	\label{eq:competition-rates}
\end{equation}
The deterministic system is $\dot u_i=u_i(b_i(u)-d_i(u))$.
The entries of $B^{\rm ch}-B^{\rm cyc}$ are non-negative and $h'(m)\geq0$.
All off-diagonal effects on the per-capita growth rates are then non-positive, and the model is competitive.
It is not a standard Lotka--Volterra model with a fixed matrix, since the competition coefficients between residents depend on $m$.

The assumptions of the paper are satisfied: the rates in~\eqref{eq:competition-rates} are non-negative and of class $C^1$. Moreover, $h(0)=0$ and $h(m)=O(m^8)$ near zero, so a single invading individual has a vanishing effect on the residents as $K$ tends to infinity. 
Finally at $m=0$, the resident system is exactly~\eqref{eq:competition-resident}, whose stable hyperbolic cycle was verified above.
Along this cycle, the fourth species, when in a small quantity, is approximated by a linear birth-and-death branching process in a periodic environment, with
\begin{equation}
	b_D\bigl(\bar n(t),0\bigr)=1,
	\qquad
	d_D\bigl(\bar n(t),0\bigr)
	=0.0588\bar x_1(t)+2.156\bar x_2(t)+0.49\bar x_3(t).
	\label{eq:competition-rare-rates}
\end{equation}
Numerically,
\begin{equation}
	f(\bar n)\simeq0.43548>0.
	\label{eq:competition-fitness}
\end{equation}
The survival probability varies between $0.00130$ and $0.77597$, depending on the phase of introduction, and $\bar s\simeq0.43548$.
The equality $\bar s=f(\bar{\mathbf n})$ is not a coincidence: since $b_D\equiv1$, it follows from Corollary~\ref{cor:average}.
Figure~\ref{fig:competition-phase} compares Kendall's formula with $q_K(\theta)$ for the full individual-based model.
We use $K=5\times10^4$ and $\varepsilon=0.002$, so that the invasion threshold is $\lfloor\varepsilon K\rfloor=100$.

The simulations give the same general phase effect as the limiting result.
At $13$ of the $16$ phases, the limiting value belongs to the $95\%$ confidence interval.
However, the agreement is not equally good at every phase.
At phases $0.625$ and $0.6875$, the estimates for finite $K$ are $0.649$ and $0.196$, while the limiting values are $0.452$ and $0.052$.
To identify the source of this gap, we introduce $q_{100}^{\rm BD}(\theta)$, where BD stands for birth--death. It is the probability that the periodic linear birth--death process, started from one individual at phase $\theta$, reaches $100$ before extinction. 
Then, $q_{100}^{\rm BD}$ uses the same threshold as $q_K$, but not the finite $K$ fluctuations of the resident population. Comparing $q_{100}^{\rm BD}$ with $q_K$ tests the finite $K$ approximation, whereas comparing it with $s$ measures the effect of replacing the threshold $100$ by survival forever. More precisely, if $u_n(\theta)$ is the probability of reaching $100$ before $0$ from $n$ individuals, then $u_n$ is the $T$-periodic solution of
\begin{equation}
	\partial_\theta u_n(\theta)
	+n b_D(\bar n(\theta),0)\bigl(u_{n+1}(\theta)-u_n(\theta)\bigr)
	+n d_D(\bar n(\theta),0)\bigl(u_{n-1}(\theta)-u_n(\theta)\bigr)=0,
	\label{eq:competition-backward}
\end{equation}
for $1\leq n\leq99$, with $u_0=0$, $u_{100}=1$, and $q_{100}^{\rm BD}(\theta)=u_1(\theta)$. At normalized phases $0.625$ and $0.6875$, we obtain $q_{100}^{\rm BD}=0.642$ and $0.207$. These values are close to $q_K=0.649$ and $0.196$, but not to $s=0.452$ and $0.052$. Hence, the gap mainly comes from the finite threshold: reaching $100$ individuals is less demanding than surviving forever. It is not primarily caused by $K=5\times10^4$ being too small.

\begin{figure}[h]
	\centering
	\includegraphics[width=0.96\textwidth]
	{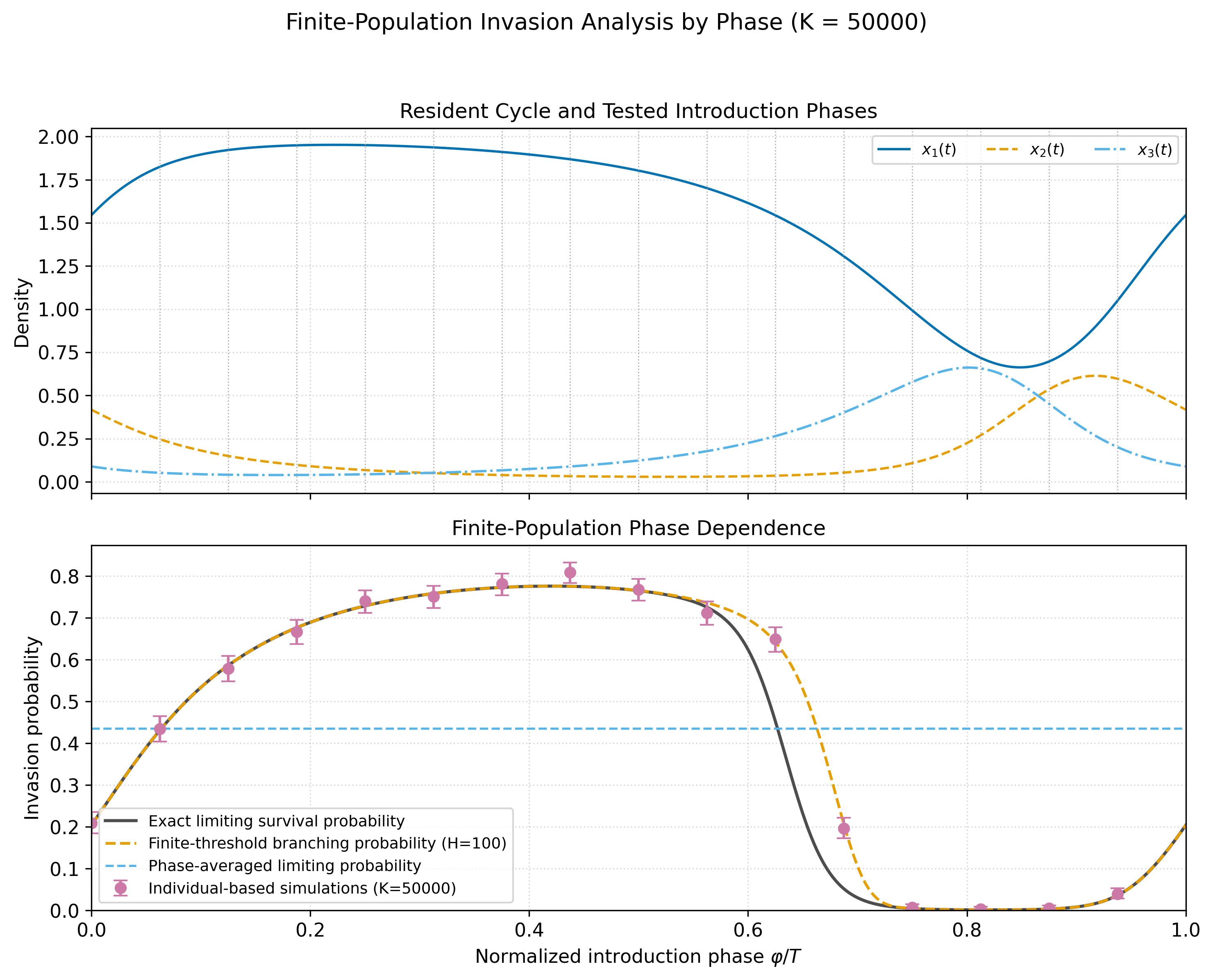}
	\caption{Invasion in a competitive Lotka-Volterra dynamics.
		The upper panel shows the resident cycle and the phases used in the simulations.
        The lower panel compares the survival probability $s(\theta)$, the probability $q_{100}^{\rm BD}(\theta)$ of reaching $100$ in the birth--death approximation, and the individual-based estimate $q_K(\theta)$ for $K=5\times10^4$.
	}
	\label{fig:competition-phase}
\end{figure}

\subsubsection{Post-establishment dynamics and phase-dependent basins}

We use $\varepsilon=0.002$ for this example.
At that density, $h(\varepsilon)\simeq2.56\times10^{-6}$, so the resident interaction matrix is still very close to $B^{\rm cyc}$ at the end of the branching stage.
To display the post-establishment attractor, we select a deterministic trajectory starting at phase $14/24$ with $m(0)=10^{-5}$. 
This small positive density is of the same order as $1/K=2\times10^{-5}$ for the value $K=5\times10^4$ used above, but it is only a deterministic indicator for rarity, not exactly one individual. 
The phase is chosen because the resulting establishment state lies in the basin of the interior attractor.
This trajectory is illustrative and is not used to estimate an invasion probability.
It reaches $\varepsilon$ after $T_\varepsilon\simeq7.7$.
The corresponding establishment phase is approximately $0.684$.

For this trajectory, the four species remain bounded away from zero. 
To quantify sensitivity to initial conditions, we follow the evolution of infinitesimal perturbations along the trajectory. The Lyapunov exponents measure their average exponential rates of growth or decline over long times. Numerically, we integrate the system linearized along the trajectory and regularly orthonormalize the perturbation vectors to keep the computation stable. The averages are evaluated only after discarding the initial transient~\cite{Vano, WX}.
The resulting spectrum is

\begin{equation}
	(\lambda_1,\lambda_2,\lambda_3,\lambda_4)
	\simeq(4.47\times10^{-3},\quad 2.32\times10^{-5},\quad
	-0.1627,\quad -0.5999).
	\label{eq:competition-lyapunov}
\end{equation}
The positive largest exponent indicates exponential sensitivity to nearby initial conditions. 
The exponent close to zero corresponds to perturbations tangent to the trajectory.
Together with the irregular persistent trajectory shown in Figure~\ref{fig:competition-chaos}, these properties support the presence of an interior chaotic attractor.
We obtain the same conclusion with smaller integration steps, longer transients, nearby initial conditions, and small changes of $B^{\rm cyc}$.
The smallest density observed on the reference attractor is $0.0235$.
These results are numerical and do not give a mathematical proof of chaos.
For this reason, we call this attractor numerically chaotic.

\begin{figure}[h]
	\centering
	\includegraphics[width=0.96\textwidth]
	{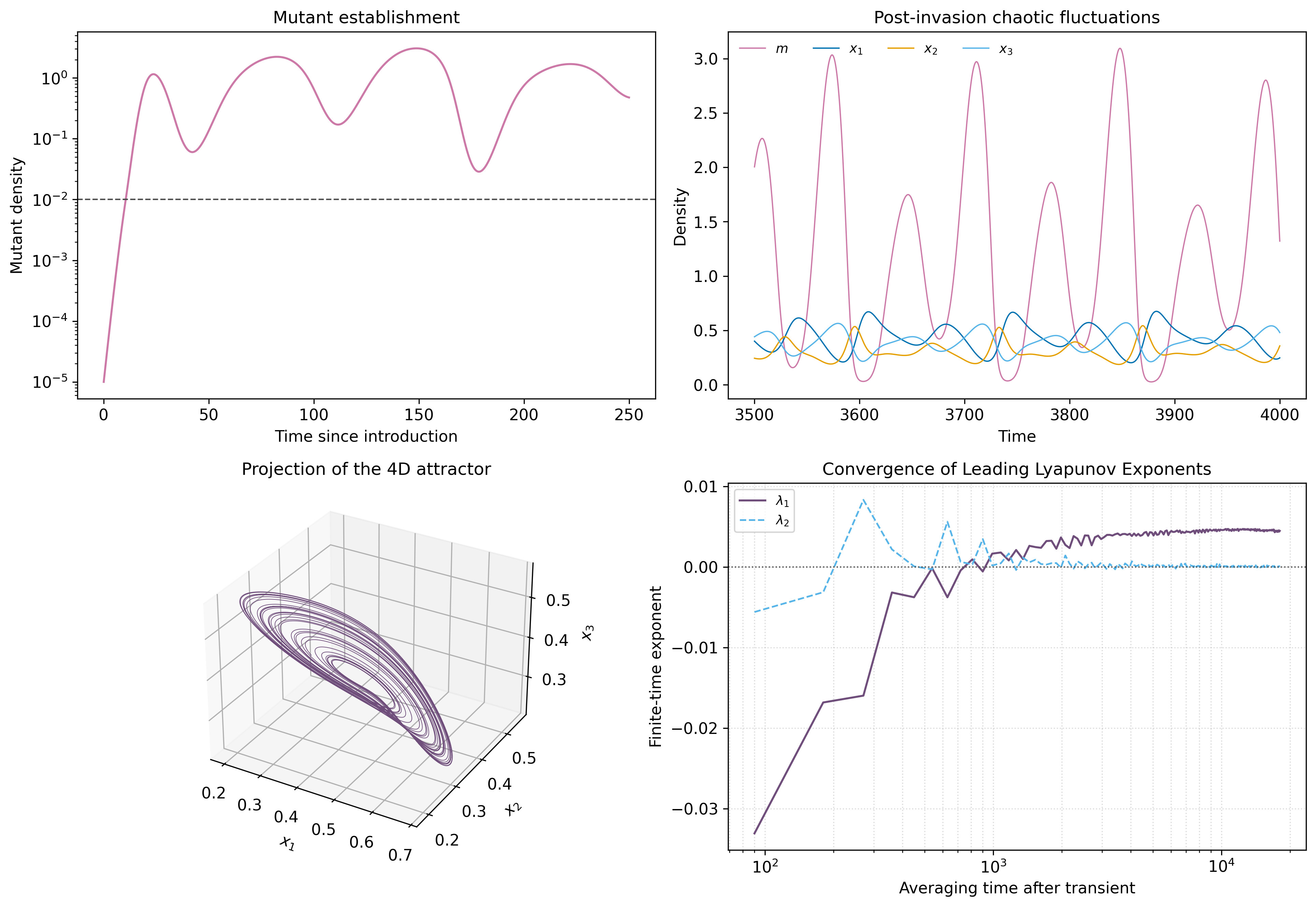}
	\caption{Competitive dynamics after establishment, for an initial condition in the basin of the interior chaotic attractor.
		The time series and the phase-space projection show irregular dynamics.
		The Lyapunov spectrum has one positive exponent, one exponent close to zero, and two negative exponents.
	}
	\label{fig:competition-chaos}
\end{figure}

The result after establishment depends on the establishment phase.
To study this effect independently of the first stochastic stage, we start the same four-dimensional deterministic system from
\begin{equation}
	u_\vartheta(0)=\bigl(\varepsilon,\bar n(\vartheta)\bigr),
	\qquad \vartheta\in[0,T).
	\label{eq:competition-basin-ics}
\end{equation}
Here, $\vartheta$ is the phase at which the fourth species has already reached $\varepsilon$.
It is not necessarily the phase at which the first individual appeared.
We use $48$ phases equally spaced in time over one resident period:
\begin{equation}
	\vartheta_k=\frac{kT}{48},\qquad k=0,\ldots,47.
	\label{eq:competition-basin-grid}
\end{equation}
The phases are equally spaced, but the corresponding points are not necessarily equally spaced in the phase portrait, because the speed along the resident cycle is not constant.
We integrate each trajectory up to time $8000$.
Four consecutive phases, $\vartheta/T\in\{0.6667,0.6875,0.7083,0.7292\}$, lead to interior numerically chaotic coexistence.
The other $44$ trajectories approach the boundary, where one or more species are absent.
Figure~\ref{fig:competition-basins} summarizes these two outcomes along the resident cycle.

The equations are the same for all $48$ simulations, but the initial condition changes with the phase.
These initial conditions can belong to different basins of attraction.
The ratio $4/48$ only describes the result on this grid.
It is not the probability of reaching the chaotic attractor.

\begin{figure}[h]
	\centering
	\includegraphics[width=0.96\textwidth]
	{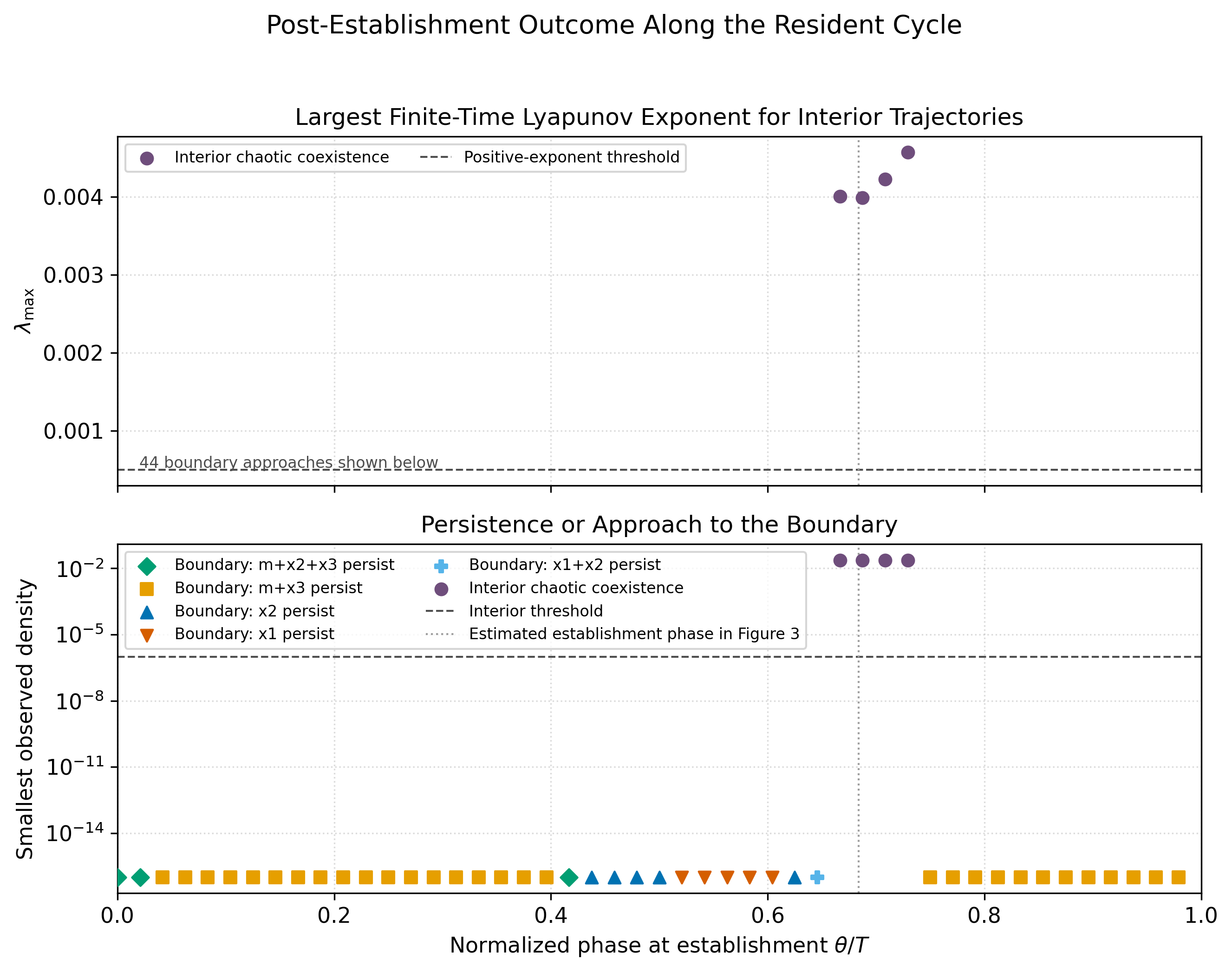}
	\caption{Long-term outcome as a function of the establishment phase $\vartheta$.
		The upper panel gives the largest Lyapunov exponent for trajectories that remain in the interior.
		The lower panel gives the smallest density and indicates which species remain on the boundary.
		The equations are the same at every phase, only the initial condition~\eqref{eq:competition-basin-ics} changes.
	}
	\label{fig:competition-basins}
\end{figure}

\subsection{Example 3: Maintenance of periodic dynamics from 2D to 3D  mutualistic system}
\label{subsec:mutualism}

Our third example is a density-dependent interaction model from the mutualistic framework of Mian et al.~\cite{MBLS}. 
Two resident types oscillate on a stable cycle that alternates between mutualistic and antagonistic regimes. 
After establishment of a mutant derived from the second resident type, all three types coexist on a new stable cycle, so periodic dynamics are maintained in a higher-dimensional system.

\subsubsection{Individual-based model and resident cycle}

We consider two resident types, $x$ and $y$, and a mutant $z$ derived from $y$.
All birth and death rates are non-negative:
\begin{equation}
	\begin{aligned}
		b_x(x,y,z) & =a_1,
		           & d_x(x,y,z) & =b_1(y+\alpha z-c_1)^2+d_1x, \\
		b_y(x,y,z) & =d_2x,
		           & d_y(x,y,z) & =a_2+b_2(y-c_2)^2+\kappa z,  \\
		b_z(x,y,z) & =d_3x,
		           & d_z(x,y,z) & =a_3+b_3(z-c_3)^2+\kappa' y.
	\end{aligned}
	\label{eq:mutual-full-rates}
\end{equation}
When $z=0$, the deterministic resident dynamics are
\begin{equation}
	\left\{
	\begin{aligned}
		\dot x & =x\left[a_1-b_1(y-c_1)^2-d_1x\right],  \\
		\dot y & =y\left[-a_2-b_2(y-c_2)^2+d_2x\right].
	\end{aligned}
	\right.
	\label{eq:mutual-resident}
\end{equation}

This system belongs to the class of density-dependent interactions studied in~\cite{MBLS}. 
The effect of $x$ on the per-capita growth rate of $y$ is the positive constant $d_2$.
The effect of $y$ on the per-capita growth rate of $x$ is $-2b_1(y-c_1)$: it is positive when $y<c_1$ and negative when $y>c_1$. 
Thus, the interaction is mutualistic below the threshold $c_1$, whereas above it the high density of $y$ becomes detrimental to $x$. The resident cycle crosses this threshold and therefore alternates between the two interaction regimes.

We use
\begin{equation}
	\begin{gathered}
		(a_1,b_1,c_1,d_1)=(13.69,0.25,4,0.36),\\
		(a_2,b_2,c_2,d_2)=(9,0.04,16,1).
	\end{gathered}
	\label{eq:mutual-resident-parameters}
\end{equation}
The positive equilibrium is unstable and surrounded by a stable periodic orbit of period $T\simeq0.46095$.
Numerically, the Floquet multipliers are $1$ and approximately $0.625$. The multiplier $1$ corresponds to motion along the cycle. The other multiplier measures perturbations away from the cycle, since $|0.625|<1$ these perturbations decrease after each period. 
Then, the resident cycle is hyperbolic and locally attracting, as required in Theorem~\ref{th_proba_inv}.

\subsubsection{Invasion of a mutant derived from \texorpdfstring{$y$}{y}}

The mutant parameters are
\begin{equation}
	(a_3,b_3,c_3,d_3,\kappa',\alpha,\kappa)
	=(9,0.04,8,1,0.05,8,0.5).
	\label{eq:mutual-mutant-parameters}
\end{equation}
When $z$ is rare, the terms containing its density can be neglected.
Its population is then approximated by a linear birth-and-death branching process in a periodic environment, with
\begin{equation}
	b_D\bigl(\bar n(t),0\bigr)=d_3\bar x(t),
	\qquad
	d_D\bigl(\bar n(t),0\bigr)=a_3+b_3c_3^2+\kappa'\bar y(t).
	\label{eq:mutual-rare-rates}
\end{equation}
We obtain
\begin{equation}
	f(\bar n)=\frac{1}{T}\int_0^T
	\left[d_3\bar x(t)-a_3-b_3c_3^2-\kappa'\bar y(t)\right]\,\mathrm dt
	\simeq0.44228>0.
	\label{eq:mutual-fitness}
\end{equation}
The survival probability varies between $0.011$ and $0.064$, and $\bar s\simeq0.0356$.
It is higher when $x$ is large, because mutant births increase, and when $y$ is small, because mutant deaths decrease.
Figure~\ref{fig:mutual-phase} compares the limiting probability with Gillespie estimates for $K=3\times10^4$ and $\varepsilon=0.01$.

\begin{figure}[h]
	\centering
	\includegraphics[width=0.96\textwidth]
	{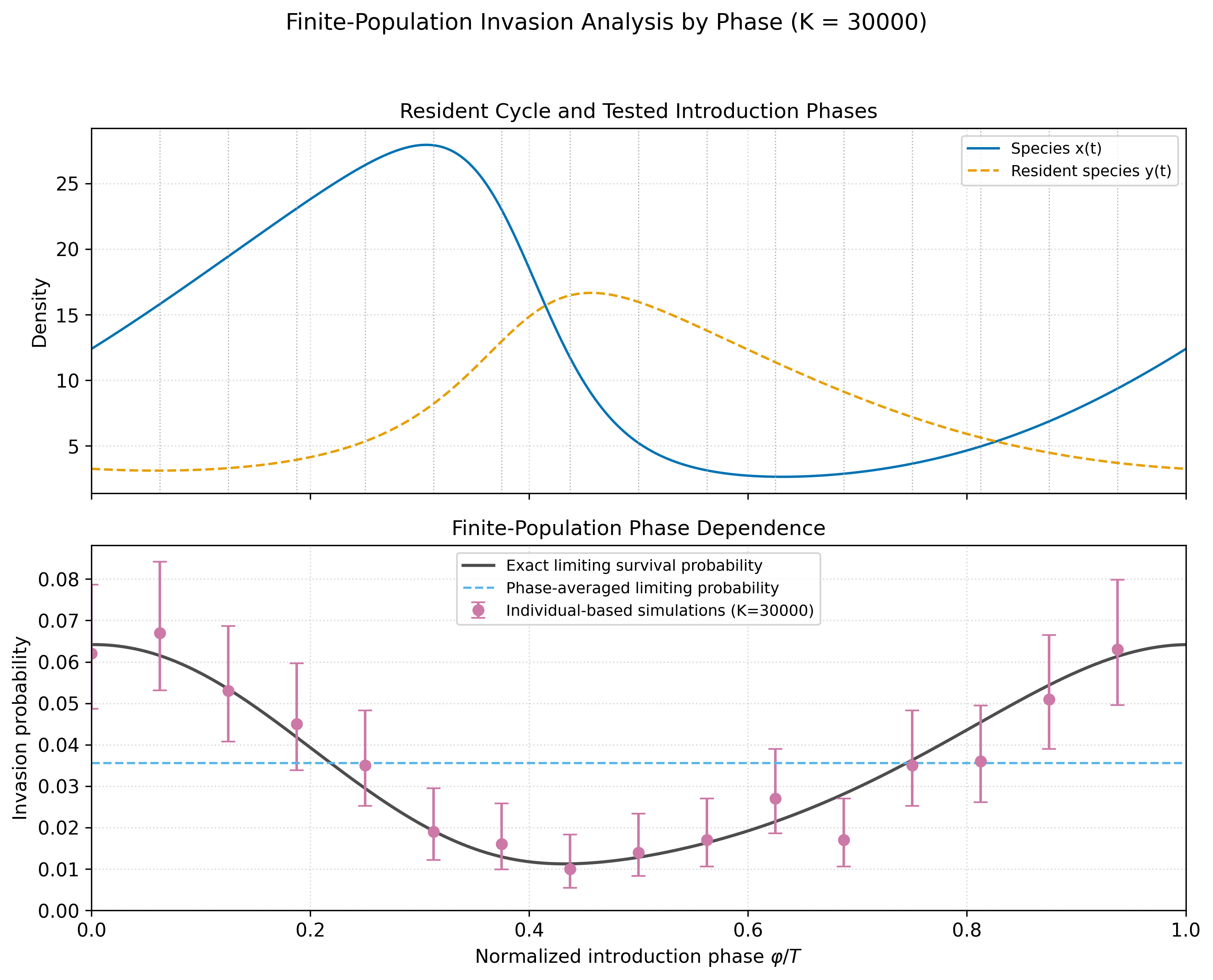}
	\caption{Mutualistic example.
		The upper panel shows the resident cycle and the phases used in the simulations.
		The lower panel compares $s(\theta)$ with the individual-based estimate $q_K(\theta)$ for $K=3\times10^4$.
	}
	\label{fig:mutual-phase}
\end{figure}

\newpage
\subsubsection{Post-establishment dynamics}

When the mutant density is no longer small, the deterministic limit of the same birth--death model~\eqref{eq:mutual-full-rates} is
\begin{equation}
	\left\{
	\begin{aligned}
		\dot x & =x\left[a_1-b_1(y+\alpha z-c_1)^2-d_1x\right],   \\
		\dot y & =y\left[-a_2-b_2(y-c_2)^2+d_2x-\kappa z\right],  \\
		\dot z & =z\left[-a_3-b_3(z-c_3)^2+d_3x-\kappa' y\right].
	\end{aligned}
	\right.
	\label{eq:mutual-full-system}
\end{equation}
Starting from a small positive mutant density, the numerical solution converges to a stable three-dimensional cycle.
The resident type $y$ is not eliminated, and the oscillations remain.
The three types $x$, $y$, and $z$ coexist on this cycle (Fig. \ref{fig:mutual-post-invasion}).

\begin{figure}[h]
	\centering
	\includegraphics[width=0.96\textwidth]
	{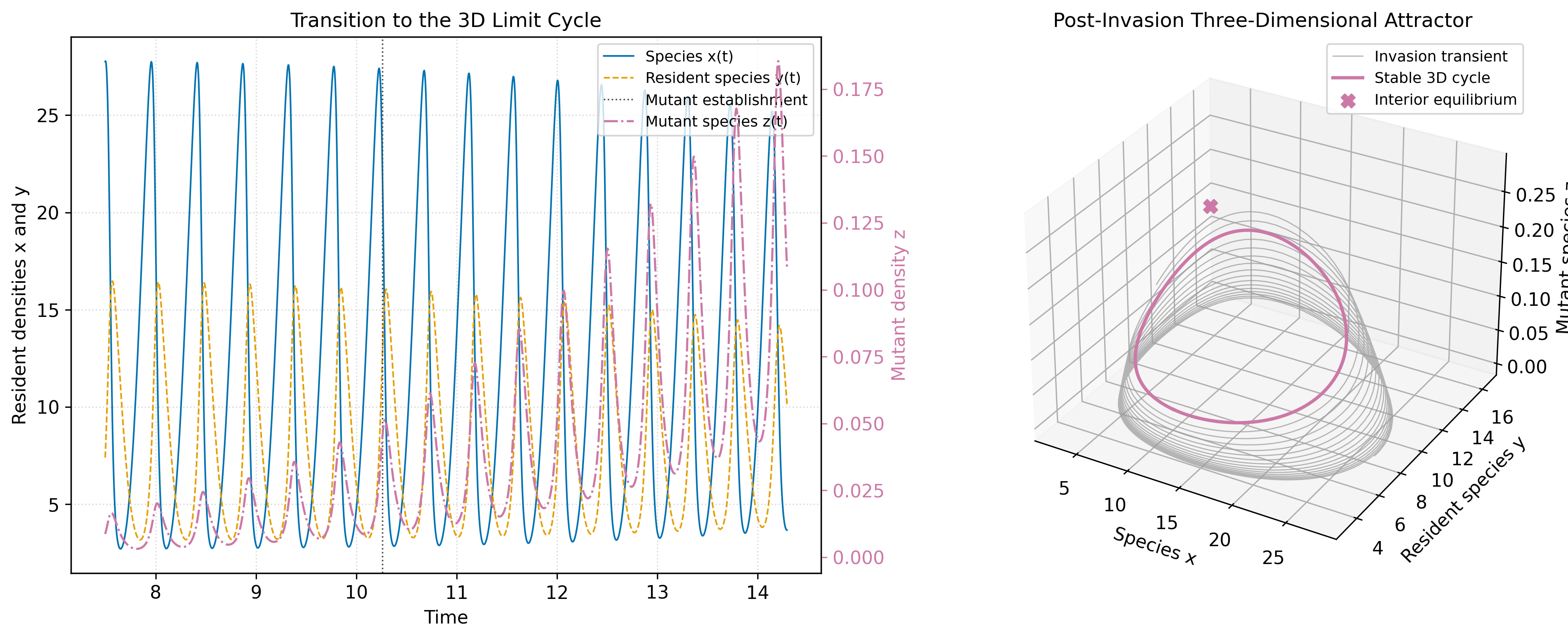}
	\caption{Mutualistic dynamics after establishment.
		The trajectory leaves the resident two-dimensional cycle and converges to a stable three-dimensional cycle.
		The three types coexist on this new cycle.
	}
	\label{fig:mutual-post-invasion}
\end{figure}

\newpage
\section{Proof of Theorem \ref{th_proba_inv}}\label{sec_proof}

Let us first recall the following notations of Section \ref{sec:results}: $\Zperp = (Z_1^K,\ldots,Z_{D-1}^K)$, $\rho_D=b_D-d_D$, $\sigma_D=b_D+d_D$, and, for $\theta\in\R$,
$\bar{\mathbf n}_\theta(t):=\bar{\mathbf n}\bigl((t+\theta)\bmod T\bigr)$ for the resident orbit
shifted by the phase $\theta$. Besides the stopping times $T_a^K$ of \eqref{defTK} we introduce the exit times
\begin{equation}\label{def_TKex}
   T_{\alpha}^{K,ex}:= \inf_{s \geq 0} \{ \Zperp(s) \notin \mathcal B_{\alpha}(\bar{\mathbf n}) \},
\end{equation}
where $\mathcal B_{\alpha}(\bar{\mathbf n})$ has been introduced in \eqref{alpha-neighb}.
We denote by $n^{x}$ the solution of \eqref{dynsyst} restricted to its $D-1$ first coordinates with initial
condition $x\in\R^{D-1}$, and we set $d(x,\bar{\mathbf n}):=\inf_{s\in[0,T]}\Vert x-\bar{\mathbf n}(s)\Vert$.

To prove Theorem \ref{th_proba_inv}, the idea is to see the mutant process $N_D^K$, until it reaches the size $0$ or $\eps K$, as a birth and death process in a variable environment, the environment being the trajectory of the resident populations. To implement this strategy we first need to control the sizes of the resident type populations before the time $T_0^K\wedge T_\eps^K$, that is during a time of order $\ln K$ if $\eps K$ is reached before $0$.

The periodicity of the resident dynamics makes the analysis more delicate than in the case of an attracting equilibrium treated in \cite{champagnat2006microscopic}. Indeed, when the deterministic system converges to an equilibrium $\bar{\mathbf n}$, the sole knowledge that $\Zperp(t)$ stays in a small neighbourhood of $\bar{\mathbf n}$ is enough to know the environment seen by the mutant with the required accuracy. This is no longer the case for a periodic orbit: staying in a neighbourhood of the cycle says nothing about the position of $\Zperp(t)$ along the cycle, and demographic fluctuations make the phase of $Z^K_\perp$ diffuse along the cycle, so that $Z^K_\perp(t)$ and $\bar{\mathbf n}(t)$ may drift apart along the cycle even though both stay in its vicinity.

What we will prove instead is that at any time $t$ the resident populations are close to $\bar{\mathbf n}(t+\theta)$ for a phase $\theta$ which itself evolves slowly, by increments of order $\eps$ at each period. The deterministic tool which makes this possible is the asymptotic phase associated with a stable hyperbolic limit cycle, which we recall in Lemma \ref{lem:phase}. Near such a cycle the flow does not amplify perturbations: it contracts them transversally and only shifts the phase. Combining this with a one-period estimate for the deviation of $\mathbf{Z}^K$ from its deterministic limit (Lemma \ref{lem:onestep}), we obtain the control of the environment that we need (Proposition \ref{prop:env}), and in particular \eqref{small_bias}.

We will then use this control in two different ways, and this is where the elementary observation made in \eqref{eq:phase-inv} comes into play. Because of \eqref{eq:phase-inv}, the phase drift, which we are only able to bound by $\mathcal O(\eps\ln K)$ and which we therefore cannot neglect on the time scale of the invasion, has no effect on the mean growth rate of the mutant population. This is what allows us to prove \eqref{extinction} and \eqref{survive_ND} on the whole  invasion time scale. The extinction probability, on the contrary, does depend on the phase at which the mutant arises, as the examples of Section \ref{sec_examples} show quantitatively, and we shall obtain \eqref{not_survive_ND} by comparing $N_D^K$ with birth and death processes in a $T$-periodic environment during a bounded number of periods, over which the phase has not had time to drift. The last ingredient is that these limiting processes are exactly solvable: being linear birth and death processes in a periodic environment, their extinction probability is given by Kendall's formula, which we state in Lemma \ref{lem:kendall}.

Let us recall the classical vocabulary attached to a periodic orbit, in the form in which it was already used in Section
\ref{sec:results}. Let $x_0:=\bar{\mathbf n}(0)$ and let $\Sigma:=\{x\in\R^{D-1}\colon \langle x-x_0,r(x_0)\rangle=0\}$ be the affine hyperplane through $x_0$ orthogonal to the velocity $r(x_0)$ of the orbit, where $r$ denotes the vector field \eqref{def_r} of the system \eqref{dynsyst} restricted to its $D-1$ first coordinates. Since $r(x_0)\neq0$, the implicit function theorem provides a neighbourhood $U_0$ of $x_0$ in $\Sigma$ and a map $\tau\colon U_0\to(0,\infty)$, as regular as the vector field and satisfying $\tau(x_0)=T$, such that $n^{x}(\tau(x))\in\Sigma$. The Poincar\'e map is
$$
   P\colon U_0\to\Sigma,\qquad P(x):=n^{x}\bigl(\tau(x)\bigr),
$$
and $x_0$ is a fixed point of $P$. The eigenvalues $\mu_2,\dots,\mu_{D-1}$ of the $(D-2)\times(D-2)$ matrix $DP(x_0)$ are the nontrivial Floquet multipliers of the cycle. Equivalently, they are the eigenvalues other than the trivial one $\mu_1=1$ of the monodromy matrix $\Phi(T)$, where $\Phi$ solves the variational equation $\dot\Phi(t)=Dr(\bar{\mathbf n}(t))\Phi(t)$, $\Phi(0)=\mathrm{Id}$. Our standing assumption that $\bar{\mathbf n}$ is a stable hyperbolic limit cycle means, as in Section \ref{sec:results}, that $|\mu_j|<1$ for every $j\ge2$.

\begin{lemma}[asymptotic phase]\label{lem:phase}
Assume that $\bar{\mathbf n}$ is a stable hyperbolic periodic orbit of period $T$ of \eqref{dynsyst} restricted to its $D-1$ first coordinates. Then there exist $\alpha_0>0$, finite constants $C_\Phi\ge 1$ and $L_\Theta$, a real number
$\lambda>0$, and a map $\Theta\colon\mathcal B_{\alpha_0}(\bar{\mathbf n})\to\R/T\Z$, called the asymptotic phase, with the following three properties.
\begin{enumerate}
\item[(i)] $\Theta$ is $L_\Theta$-Lipschitz:
\begin{equation}\label{eq:Theta-Lip}
   |\Theta(x)-\Theta(y)|\le L_\Theta\Vert x-y\Vert,
   \qquad x,y\in\mathcal B_{\alpha_0}(\bar{\mathbf n}),
\end{equation}
the left-hand side belonging to $\R/T\Z$.
\item[(ii)] Every trajectory starting in $\mathcal B_{\alpha_0}(\bar{\mathbf n})$ becomes
exponentially close to the point of the cycle prescribed by its phase: for every $x\in\mathcal B_{\alpha_0}(\bar{\mathbf n})$ and every $t\ge0$,
\begin{equation}\label{eq:asymptotic-phase}
  \bigl\Vert n^{x}(t)-\bar{\mathbf n}\bigl(t+\Theta(x)\bigr)\bigr\Vert
   \;\le\; C_\Phi\, e^{-\lambda t}\, d(x,\bar{\mathbf n}).
\end{equation}
\item[(iii)] For every $s,t \geq 0$, $x\in\mathcal B_{\alpha_0}(\bar{\mathbf n})$, $\Theta(\bar{\mathbf n}(s))=s$ and $\Theta(n^{x}(t))=\Theta(x)+t$ in $\R/T\Z$.
\end{enumerate}
Moreover $\lambda$ may be taken to be any positive number smaller than $-\frac1T\ln\max_{j\ge2}|\mu_j|$.
\end{lemma}

This is the classical asymptotic phase theorem for stable hyperbolic limit cycles. The different properties stated in Lemma \ref{lem:phase} may be found in \cite{Hartman, Guckenheimer75, Chicone, kuznetsov} for instance.  
Since the vector field is of class $\mathcal C^1$, $\Theta$ is of class $\mathcal C^1$ on a tubular neighbourhood of the orbit; for $\alpha_0$ small enough, $B_{\alpha_0}(\bar{\mathbf n})$ is contained in such a neighbourhood, and \eqref{eq:Theta-Lip} follows by compactness.
The constant $C_\Phi\ge1$ accounts for the transient behaviour: the convergence in \eqref{eq:asymptotic-phase} is exponential, but the distance to the cycle may first increase by the bounded factor $C_\Phi$ before decreasing.

Replacing if necessary $T$ by an integer multiple $mT$ (which changes nothing by \eqref{eq:phase-inv}), we assume from now on
\begin{equation}\label{eq:contraction}
  \kappa:=C_\Phi e^{-\lambda T}<\tfrac12 .
\end{equation}

Let us stress that this replacement is not only cosmetic, since $C_\Phi \geq 1$ and $\mu_2$ may be close to one, the contraction of the Poincaré map over a single period can be arbitrarily weak, and several periods may be needed. This replacement changes the constant $D_1$ in \eqref{eq:proba-E} into $8(D-1)/(mT)$, and the times $nT$ in \eqref{small_bias_phase} into $nmT$; we keep the notation $T$ for simplicity.

\begin{lemma}[one--period deviation]\label{lem:onestep}
There exist $\eta_0>0$ and finite positive constants $C_0,c_0$, depending only on $T$, $\alpha_0$ and the rates $(b_i,d_i, i\leq D)$, such that for every $\eta\le\eta_0$, $K>0$ and $z\in \mathcal B_{\alpha_0}(\bar{\mathbf n})\times[0,\eta]$,
\begin{equation}\label{eq:onestep}
  \P_{z}\Bigl(\ \sup_{t\leq T\wedge T^{K}_{\eta}\wedge T_{\alpha_0}^{K,ex}}
       \bigl\|\Zperp(t)-n^{z_\perp}(t)\bigr\|> C_0\eta \Bigr)\ \leq 4 (D-1) e^{-c_0 K\eta^{2}},
\end{equation}
where $z_\perp$ denotes the $D-1$ first coordinates of $z$ and $n^{z_\perp}$ is the solution of \eqref{dynsyst} restricted to its $D-1$ first coordinates, started at $z_\perp$.
\end{lemma}

Let us stress that $\eta$ bounds the mutant density and the size of the deviation, but not the distance from $z_\perp$ to the cycle, which is only required to be at most $\alpha_0$. This decoupling is what makes the induction of Proposition~\ref{prop:env} work: there, the process is only known to lie within distance $c_1\eps$ of the cycle, with $c_1>1$, whereas the deviation produced over one period must remain of size $C_0\eps$, with a constant $C_0$ that does not depend on $c_1$.

\begin{proof}
Denoting $\beta_i(z):=z_ib_i(z)$ and $\delta_i(z):=z_id_i(z)$, we may write $\Zperp$ as the strong solution to the following SDE, for $i\le D-1$ and $t\le T\wedge T_\eta^K\wedge T_{\alpha_0}^{K,ex}$
$$
  Z_i^K(t)=Z_i^K(0)+\int_0^t\bigl(\beta_i-\delta_i\bigr)(\mathbf{Z}^K(s))\,ds
  +\frac1K \widetilde Q^{+}_i\Bigl(K\!\int_0^t\!\beta_i(\mathbf{Z}^K(s))ds\Bigr)
  -\frac1K \widetilde Q^{-}_i\Bigl(K\!\int_0^t\!\delta_i(\mathbf{Z}^K(s))ds\Bigr),
$$
with $\widetilde Q^{\pm}_i$ independent compensated rate $1$ Poisson processes.  Note that the initial distance from $z_\perp$ to the cycle plays no role below: both $Z_\perp$ and $n^{z_\perp}$ start at $z_\perp$, so that the Gr\"onwall estimate is applied with a vanishing initial error. Before time $T\wedge T_\eta^K\wedge T_{\alpha_0}^{K,ex}$, $\mathbf{Z}^K$ stays in $\overline{\mathcal B_{\alpha_0}(\bar{\mathbf n})}\times[0,\eta]$ (as $N^K_D$ has jumps $\pm1$, $Z^K_D(s)<\eta$ for $s<T^K_\eta$), and, by \eqref{eq:asymptotic-phase}, $n^{z_\perp}$ stays in $\overline{\mathcal B_{C_\Phi\alpha_0}(\bar{\mathbf n})}$. On the compact set $\big(\overline{\mathcal B_{C_\Phi\alpha_0}(\bar{\mathbf n})}\cap\R_+^{D-1}\big)\times[0,\eta_0]$, $\beta_i,\delta_i$ are bounded by some $\bar\beta<\infty$, and $\beta_i-\delta_i$ is Lipschitz with some constant $\mathcal M$ in its first $D-1$ variables and with some constant $\mathcal K$ in its last one. Writing 
$$(\beta_i-\delta_i)(\mathbf{Z}^K)=(\beta_i-\delta_i)(\Zperp,0)
+O(Z_D^K)=(\beta_i-\delta_i)(\Zperp,0)+O(\eta)$$
and applying Gr\"onwall's lemma on the time interval  $[0,T]$ gives, with the notation $$\vartheta_K:=T\wedge T_\eta^K\wedge T_{\alpha_0}^{K,ex},$$
$$
  \sup_{t\le \vartheta_K}\|\Zperp(t)-n^{z_\perp}(t)\|
  \ \le\ e^{\mathcal M T}\Bigl(\mathcal K T\eta+ 2\max_{i\le D-1}\ \max_{*\in\{+,-\}}
     \ \sup_{t\le \vartheta_K}\bigl|M^{*}_i(t)\bigr|\Bigr),
$$
where, for $*=+$,
$$
  M^{+}_i(t):=\frac1K\,\widetilde Q^{+}_i\bigl(A_i(t)\bigr),
  \qquad A_i(t):=K\!\int_0^t\!\beta_i(\mathbf{Z}^K(s))\,ds,
$$
and similarly for $*=-$ with $\delta_i$ in place of $\beta_i$. On the time interval $[0,\vartheta_K]$ the rates are bounded by $\bar\beta$, so that the time change $A_i$, which is continuous and nondecreasing, takes its values in $[0,\Lambda]$ with $\Lambda:=K\bar\beta T$. Consequently
$$
  \sup_{t\le \vartheta_K}\big|M^{*}_i(t)\big|
    \leq \frac1K\sup_{u\leq \Lambda}\big|\widetilde Q^{*}_i(u)\big| ,
$$
and it only remains to control the fluctuations of a rate $1$ compensated Poisson process $\widetilde Q$ over the deterministic interval $[0,\Lambda]$.

Let $(\mathcal G_u)_{u\ge0}$ be the natural filtration of $\widetilde Q$, and let $x>0$. For $\ell\ge0$ and $\theta\in\R$, $\E[e^{\theta\widetilde Q(\ell)}]=e^{\ell(e^\theta-1-\theta)}$, since $\widetilde Q(\ell)+\ell$ is Poisson distributed with parameter $\ell$. For $0\le s\le u$, the independence and stationarity of the increments of $\widetilde Q$ thus give
$$   \E\big[e^{\theta\widetilde Q(u)}\,\big|\,\mathcal G_s\big]=e^{\theta\widetilde Q(s)}\,e^{(u-s)(e^\theta-1-\theta)}\ \ge\ e^{\theta\widetilde Q(s)},$$
because $e^\theta-1-\theta\ge0$. Hence $(e^{\theta\widetilde Q(u)})_{u\ge0}$ is a non-negative c\`adl\`ag submartingale, and Doob's maximal inequality $\P(\sup_{u\le\Lambda}X_u\ge c)\le\E[X_\Lambda]/c$, valid for such a submartingale $X$ and $c>0$, yields, for $\theta>0$,
$$   \P\Big(\sup_{u\le\Lambda}\widetilde Q(u)\ge x\Big)
  =\P\Big(\sup_{u\le\Lambda}e^{\theta\widetilde Q(u)}\ge e^{\theta x}\Big)
  \le e^{-\theta x}\,\E\big[e^{\theta\widetilde Q(\Lambda)}\big]
  =\exp\big(-\theta x+\Lambda(e^\theta-1-\theta)\big).$$
We choose $\theta=\ln(1+v)$, with $v:=x/\Lambda$. The exponent then equals $-x\ln(1+v)+\Lambda\big(v-\ln(1+v)\big)=-\Lambda h(v)$, where $h(v):=(1+v)\ln(1+v)-v$. Let $k(v):=\frac{v^2}{2(1+v/3)}=\frac32\big(v-3+\frac9{3+v}\big)$. We have $h(0)=k(0)=0$, $h'(v)=\ln(1+v)$ and $k'(v)=\frac32\big(1-\frac9{(3+v)^2}\big)$, which both vanish at $0$, and
$$  h''(v)=\frac1{1+v}\ \ge\ \frac{27}{(3+v)^3}=k''(v),\qquad v\ge0,$$
since $(3+v)^3=27+27v+9v^2+v^3\ge27(1+v)$. Integrating twice gives $h\ge k$ on $\R_+$, hence $\Lambda h(x/\Lambda)\ge x^2/(2(\Lambda+x/3))$ and
$$   \P\Big(\sup_{u\le\Lambda}\widetilde Q(u)\ge x\Big)\le\exp\Big(-\frac{x^2}{2(\Lambda+x/3)}\Big).$$
Similarly, for $\theta>0$, $(e^{-\theta\widetilde Q(u)})_{u\ge0}$ is a non-negative c\`adl\`ag submartingale, and Doob's maximal inequality gives
$$   \P\Big(\inf_{u\le\Lambda}\widetilde Q(u)\le -x\Big)
  \le\exp\big(-\theta x+\Lambda(e^{-\theta}-1+\theta)\big)
  \le\exp\big(-\theta x+\Lambda\theta^2/2\big),$$
where we used $e^{-\theta}-1+\theta\le\theta^2/2$: the difference $\theta^2/2-(e^{-\theta}-1+\theta)$ vanishes at $0$ and has derivative $\theta+e^{-\theta}-1\ge0$. Choosing $\theta=x/\Lambda$ gives the bound $\exp(-x^2/(2\Lambda))\le\exp(-x^2/(2(\Lambda+x/3)))$. Altogether,
\begin{equation}\label{eq:Poisson-max}
  \P\Big(\sup_{u\le\Lambda}\big|\widetilde Q(u)\big|\ge x\Big)\le2\exp\Big(-\frac{x^2}{2(\Lambda+x/3)}\Big),\qquad x>0.
\end{equation}
Taking $x=K\eta$ and $\Lambda=K\bar\beta T$ in \eqref{eq:Poisson-max}, we obtain, since $\eta\le\eta_0$,
$$  \P\Big(\sup_{t\leq \vartheta_K}\big|M^{*}_i(t)\big|\geq \eta\Big)
    \leq 2\exp\Big(-\frac{K \eta^{2}}{2(\bar\beta T+\eta_0/3)}\Big).$$
A union bound over the $2(D-1)$ martingales $M^{*}_i$ ($i\le D-1$, $*\in\{+,-\}$) then gives \eqref{eq:onestep}, with
$$    C_0:=e^{\mathcal M T}(\mathcal K T+2)
   \qquad\text{and}\qquad
   c_0:=\big(2(\bar\beta T+\eta_0/3)\big)^{-1}.$$
\end{proof}

\begin{proposition}[control of the environment on a time $\mathcal O(\ln K)$]
\label{prop:env}
There exist $\eps_0>0$ and finite positive constants $C_*,c_*$ such that the following holds for every $\eps\le\eps_0$, $C<\infty$ and $K$ large enough. Assume $\Zperp(0)\in\mathcal B_\eps(\bar{\mathbf n})$ and $Z_D^K(0)\le\eps$, and set
$\tau_K:=T_\eps^K\wedge C\ln K$ and $\theta_n:=\Theta\bigl(\Zperp(nT)\bigr)$. Then there is an event $\mathcal E_K=\mathcal E_K(\eps,C)$ with
\begin{equation}\label{eq:proba-E}
  \P\bigl(\mathcal E_K^{\,c}\bigr) \leq D_1 C\ln K e^{-c_*K\eps^{2}},
  \qquad D_1:=\frac{8(D-1)}{T},
\end{equation}
on which, for every integer $n$ with $nT\le\tau_K$,
\begin{equation}\label{eq:env-i}
\sup_{t\in[nT,(n+1)T\wedge\tau_K]}
\big\|Z^K_\perp(t)-\bar{\mathbf n}(t-nT+\theta_n)\big\|\le C_*\varepsilon,
\end{equation}
and, for every integer $n$ with $(n+1)T\le\tau_K$,
 \begin{align} &\big|\theta_{n+1}-\theta_n\big|\leq C_*\eps,
  \label{eq:env-ii}\\[2pt]
  &\Big|\int_{nT}^{(n+1)T}\rho_D\big(\mathbf{Z}^K(s)\big)ds-T f(\bar{\mathbf n})\Big|\leq C_*\eps .
  \label{eq:env-iii}
\end{align}
In particular $\Zperp(t)\in\mathcal B_{C_*\eps}(\bar{\mathbf n})$ for every $t\leq\tau_K$.
\end{proposition}

\begin{proof}
Let us first fix the constants. Recall $C_\Phi$ and $L_\Theta$ from Lemma \ref{lem:phase}, $C_0$ and $c_0$ from Lemma \ref{lem:onestep}, $\kappa$ from \eqref{eq:contraction} and let $L_\rho$ be a Lipschitz constant for $\rho_D$ on the compact set $\overline{\mathcal B_{\alpha_0}(\bar{\mathbf n})}\times[0,\eta_0]$. We set
\begin{equation}\label{eq:choice-Cstar}
  c_1:=\frac{C_0}{1-\kappa} \vee 1,
  \qquad
  C_*:=\big(C_0+C_\Phi c_1\big) \vee L_\Theta C_0
        \vee\ TL_\rho\big(C_0+C_\Phi c_1+1\big) ,
\end{equation}
and $c_*:=c_0$, and we take $\varepsilon_0:=\alpha_0/\big(2(c_1\vee C_*)\big)\wedge\eta_0$, so that
$c_1\varepsilon$ and $C_*\varepsilon$ are smaller than $\alpha_0$ for $\eps\leq\eps_0$. The constant $c_1$ will bound the distance to the cycle at the times $nT$, while $C_*$ bounds the three quantities appearing in \eqref{eq:env-i}--\eqref{eq:env-iii}. Note that $c_1\ge1$, so that the hypothesis $\Zperp(0)\in\mathcal B_\eps(\bar{\mathbf n})$ gives $u_0\le c_1\eps$, where we have denoted $u_n:=d\big(\Zperp(nT),\bar{\mathbf n}\big)$.

The first step of the proof consists in showing that for all $n$ such that $nT \leq \tau_K$, $u_n \leq c_1 \eps$. We will proceed by induction.

\medskip\noindent\emph{The induction.} Let $N:=\lfloor C\ln K/T\rfloor$ and, for
$0\le n\le N$, let $G_n$ be the event
$$
   G_n:=\Big\{\sup_{t\le T\wedge\sigma_n}
     \big\Vert\Zperp(nT+t)-n^{\Zperp(nT)}(t)\big\Vert\leq C_0\eps\Big\},$$
where
\[   \sigma_n:=(\tau_K-nT)^{+}\wedge
     \inf\big\{t\geq 0: \Zperp(nT+t)\notin\mathcal B_{\alpha_0}(\bar{\mathbf n})\big\},
     \numberthis\label{sigmandef}
\]
the two stopping times being exactly those appearing in Lemma \ref{lem:onestep}, shifted by $nT$: indeed $\tau_K\le T_\eps^K$, and the second one is the exit time $T^{K,ex}_{\alpha_0}$ of the shifted process. Note that $G_n$ is automatically realized when $nT\ge\tau_K$, since then $\sigma_n=0$ and the supremum reduces to its value at $t=0$, which is $0$.

Set $\mathcal E_K:=\bigcap_{n=0}^{N}G_n$. We will prove that on $\mathcal E_K$ one has $u_n\le c_1\eps$ for every $n$ with $nT\le\tau_K$. This holds for $n=0$, and we assume it for some $n$ with $(n+1)T\le\tau_K$. Since $u_n\le c_1\eps\le\alpha_0$ and, by definition of $T_\eps^K$, $Z_D^K(t)\le\eps$ for $t\le\tau_K$, the point $\mathbf{Z}^K(nT)$ belongs to $\mathcal B_{\alpha_0}(\bar{\mathbf n})\times[0,\eps]$, so that Lemma \ref{lem:onestep} may be applied at time $nT$ with $\eta=\eps$. This is precisely the event $G_n$. On the other hand, Lemma \ref{lem:phase} applied to the deterministic solution started at $\Zperp(nT)$, together with $\bar{\mathbf n}(T+\theta_n)=\bar{\mathbf n}(\theta_n)$,
gives
\begin{equation}\label{eq:contract-one-step}
  \big\Vert n^{\Zperp(nT)}(t)-\bar{\mathbf n}(t+\theta_n)\big\Vert
   \leq C_\Phi e^{-\lambda t}u_n\leq C_\Phi u_n,\quad t\geq 0.\end{equation}
 In particular, when $t=T$ we obtain according to \eqref{eq:contraction}, 
\begin{equation}\label{eq:contract-one-step-bis}  \big\Vert n^{\Zperp(nT)}(T)-\bar{\mathbf n}(\theta_n)\big\Vert\leq\kappa u_n <\frac{1}{2}u_n .\end{equation}
Combining the two estimates by the triangle inequality,
$$
  u_{n+1} \leq \bigl\Vert\Zperp((n+1)T)-n^{\Zperp(nT)}(T)\big\Vert
   +\big\Vert n^{\Zperp(nT)}(T)-\bar{\mathbf n}(\theta_n)\big\Vert
   \leq C_0\eps+\kappa\,u_n\ \le\ C_0\eps+\kappa c_1\eps \leq c_1\eps ,
$$
the last inequality being the definition of $c_1$ in \eqref{eq:choice-Cstar}. Let us check, before going on, that the exit time appearing in $\sigma_n$ in~\eqref{sigmandef} is not active, so that $\sigma_n=(\tau_K-nT)^{+}$ on the event under consideration. If it were, that is if $\varsigma:=\inf\{t\ge0:\Zperp(nT+t)\notin\mathcal B_{\alpha_0}(\bar{\mathbf n})\}$ were smaller than $T\wedge(\tau_K-nT)^{+}$, then $\varsigma\le\sigma_n$ and $G_n$ together with \eqref{eq:contract-one-step} would give $d(\Zperp(nT+\varsigma),\bar{\mathbf n})\le(C_0+C_\Phi c_1)\eps\le C_*\eps<\alpha_0$, whereas by definition of $\varsigma$ this distance is at least $\alpha_0$. Hence $\sigma_n=(\tau_K-nT)^{+}$, and the estimates below hold on the whole of $[nT,(n+1)T\wedge\tau_K]$.

\medskip\noindent\emph{Proof of \eqref{eq:env-i}--\eqref{eq:env-iii}.} All the
inequalities of this part are pathwise: we fix once and for all an
$\omega\in\mathcal E_K$ and every estimate below holds for that $\omega$. Let $n$ be such that $(n+1)T\le\tau_K$ and let
$t\in[nT,(n+1)T]$. Since $\omega$ belongs to $G_n$, the definition of the latter gives 
$$\Vert\Zperp(t)-n^{\Zperp(nT)}(t-nT)\Vert\le C_0\eps.$$ 
Combining this with \eqref{eq:contract-one-step} and with $u_n\le c_1\eps$, established in the previous part, the triangle inequality yields
$$
  \big\Vert\Zperp(t)-\bar{\mathbf n}(t-nT+\theta_n)\big\Vert
   \leq C_0\eps+C_\Phi u_n\leq \big(C_0+C_\Phi c_1\big)\eps\leq C_*\eps ,
$$
which is \eqref{eq:env-i}. For \eqref{eq:env-ii}, the cocycle identity $\Theta(n^{x}(T))=\Theta(x)+T=\Theta(x)$ in $\R/T\Z$ gives $\Theta\bigl(n^{\Zperp(nT)}(T)\bigr)=\theta_n$, so that, both points lying in $\mathcal B_{\alpha_0}(\bar{\mathbf n})$,
$$
  |\theta_{n+1}-\theta_n|
   =\big|\Theta(\Zperp((n+1)T))-\Theta\big(n^{\Zperp(nT)}(T)\big)\big|
   \leq L_\Theta\bigl\Vert\Zperp((n+1)T)-n^{\Zperp(nT)}(T)\big\Vert
   \leq L_\Theta C_0\eps\leq C_*\eps .
$$
Finally, for \eqref{eq:env-iii}, we use \eqref{eq:env-i}, which we have just established for the same $\omega$, together with $Z^K_D(t)\le\eps$ and the Lipschitz continuity of $\rho_D$:
$$
  \Big|\rho_D\big(\mathbf{Z}^K(t)\big)-\rho_D\big(\bar{\mathbf n}(t-nT+\theta_n),0\big)\Big|
   \leq L_\rho\Big(\big\Vert\Zperp(t)-\bar{\mathbf n}(t-nT+\theta_n)\big\Vert+Z_D^K(t)\Big)
   \leq L_\rho\big(C_0+C_\Phi c_1+1\big)\eps ,
$$
and integrating over one period,
$$
  \Big|\int_{nT}^{(n+1)T}\rho_D\big(\mathbf{Z}^K(s)\big)ds
   -\int_0^{T}\rho_D\big(\bar{\mathbf n}(u+\theta_n),0\big)du\Big|
   \leq TL_\rho\big(C_0+C_\Phi c_1+1\big)\eps\leq C_*\eps .
$$
The remaining integral equals $Tf(\bar{\mathbf n})$ for every value of $\theta_n$, by the phase invariance \eqref{eq:phase-inv}.

\medskip\noindent\emph{Proof of \eqref{eq:proba-E}.} It remains to bound the probability of $\mathcal E_K^{\,c}=\bigcup_{n=0}^{N}G_n^{\,c}$. We decompose $\mathcal E_K^{\,c}$ according to the first index at which the induction fails: setting $H_n:=\bigcap_{m<n}G_m$, with $H_0:=\Omega$, we have the disjoint union
\begin{equation}\label{eq:first-failure}
   \mathcal E_K^{\,c}=\bigcup_{n=0}^{N}\bigl(G_n^{\,c}\cap H_n\bigr),
   \qquad\text{hence}\qquad
   \P\bigl(\mathcal E_K^{c}\bigr)=\sum_{n=0}^{N}\P\big(G_n^{c}\cap H_n\big).
\end{equation}
Let us fix $n\le N$. The event $H_n$ is $\mathcal F_{nT}$-measurable, since $G_m$ only involves the process up to time $(m+1)T\le nT$ for $m<n$, and so is $\{nT<\tau_K\}$, $\tau_K$ being a stopping time. On $H_n\cap\{nT\ge\tau_K\}$
the event $G_n$ is realized, as noticed above, so that 
$$ \P\big(G_n^{c}\cap H_n\cap\{nT\ge\tau_K\}\big)=0. $$
On $H_n\cap\{nT<\tau_K\}$, the induction of the previous part gives $u_n\leq c_1\eps\le\alpha_0$, and $Z_D^K(nT)\le\eps$ by definition of $T_\eps^K$; hence $\mathbf{Z}^K(nT)\in\mathcal B_{\alpha_0}(\bar{\mathbf n})\times[0,\eps]$, and the
strong Markov property at time $nT$ together with Lemma \ref{lem:onestep} applied with $\eta=\eps$ yield
$$
   \P\big(G_n^{\,c}\mid\mathcal F_{nT}\big)
   \ \le\ 4(D-1)\,e^{-c_0K\eps^{2}}
   \qquad\text{on }H_n\cap\{nT<\tau_K\} .
$$
Taking expectations, $\P(G_n^{\,c}\cap H_n)\le 4(D-1)e^{-c_0K\eps^{2}}$ for every $n$. Plugging this into \eqref{eq:first-failure} and using $N+1\le 2C\ln K/T$, which holds as soon as $\ln K\ge T/C$, we obtain \eqref{eq:proba-E} with $D_1=8(D-1)/T$ and $c_*=c_0$.
\end{proof}

We will now state properties of branching processes living in a deterministic periodic environment. Together with the control of the distance between $\Zperp$ and the deterministic periodic orbit we have just established in the previous lemmas, it will allow us to come back to the control of our actual population process.

\begin{lemma}[Kendall's formula in a periodic environment]\label{lem:kendall}
Let $\beta,\delta:\R_+\to\R_+$ be continuous and $T$-periodic, let $\rho:=\beta-\delta$ and $f:=\frac1T\int_0^T\rho(s)ds$. For $\theta\in[0,T)$ let $P^\theta$ be the linear birth-and-death process with individual birth rate $\beta(\theta+t)$ and death rate $\delta(\theta+t)$ at time $t$, started from $P^\theta(0)=1$, let $q(\theta):=\P(\exists t\ge0: P^\theta(t)=0)$ be its extinction probability, and let $\mathscr{P}^\theta:=(P^\theta(nT))_{n\in\N}$ be the associated discrete time Galton--Watson process, with generating function $g_\theta$.
\begin{enumerate}
\item[(a)] If $f\le0$ then $q(\theta)=1$ for every $\theta$.
\item[(b)] If $f>0$ then $q(\theta)<1$ and
\begin{equation}\label{eq:kendall}
  1-q(\theta)
  =\frac{1-e^{-Tf}}
        {1-e^{-Tf}+\displaystyle\int_0^{T}\delta(\theta+t)\,
          e^{-\int_0^{t}\rho(\theta+u)du}dt}
  =\frac{1-e^{-Tf}}
        {\displaystyle\int_0^{T}\beta(\theta+t)\,
          e^{-\int_0^{t}\rho(\theta+u)du}dt}.
\end{equation}
\item[(c)] If $f>0$ then $\mathscr{P}^\theta$ is a linear fractional Galton--Watson process and
\begin{equation}\label{eq:multiplier}
  g_\theta'\bigl(q(\theta)\bigr)=\frac{1}{g_\theta'(1)}=e^{-Tf} .
\end{equation}
  Consequently, for every $n\in\N$ and every real $A\ge1$,
\begin{equation}\label{eq:LF-rates}
   0\ \le\ q(\theta)-\P\bigl(\mathscr{P}^\theta(n)=0\bigr)\ \le\ e^{-nTf}
   \qquad\text{and}\qquad
   \P\bigl(0<\mathscr{P}^\theta(n)<A\bigr)\ \le\ A\,e^{-nTf} .
\end{equation}
\end{enumerate}
\end{lemma}

\begin{proof}
The starting point is Kendall's explicit solution \cite{Kendall1948} for a linear birth-and-death process with time-dependent rates: setting
$$
   W_\theta(t)=\int_0^t\delta(\theta+s)\,e^{-\int_0^s\rho(\theta+u)du}ds,
   \qquad
   V_\theta(t)=\int_0^t\beta(\theta+s)\,e^{-\int_0^s\rho(\theta+u)du}ds,
$$
the law of $P^\theta(t)$ has an atom at $0$ of mass $W_\theta(t)/(1+W_\theta(t))$ and, conditionally on $\{P^\theta(t)>0\}$, is geometric on $\{1,2,\dots\}$ with parameter $1/\big(1+e^{R_\theta(t)}V_\theta(t)\big)$, where $R_\theta(t):=\int_0^t\rho(\theta+u)\,du$. In particular $s\mapsto\E[s^{P^\theta(t)}]$ is a linear fractional function of $s$. 

(a) and (b). By periodicity and \eqref{eq:phase-inv},
$$e^{-\int_0^{kT+t}\rho(\theta+u)du}=e^{-kTf}e^{-\int_0^{t}\rho(\theta+u)du},$$
so that
$$
  W_\theta(\infty)=\sum_{k\geq 0}e^{-kTf}\int_0^T\delta(\theta+t)
   e^{-\int_0^t\rho(\theta+u)du}dt
  =\frac{1}{1-e^{-Tf}}\int_0^T\delta(\theta+t)
   e^{-\int_0^t\rho(\theta+u)du}dt
$$
if $f>0$, and $W_\theta(\infty)=\infty$ if $f\le0$, the general term of the series not tending to $0$. Since 
$$q(\theta)=\lim_{t\to\infty}\P(P^\theta(t)=0) =\frac{W_\theta(\infty)}{1+W_\theta(\infty)},$$ 
this gives (a) and the first expression in \eqref{eq:kendall}; the second one follows from
$$
  \int_0^T\rho(\theta+t)e^{-\int_0^t\rho(\theta+u)du}dt
  =\Bigl[-e^{-\int_0^t\rho(\theta+u)du}\Bigr]_0^T=1-e^{-Tf}.
$$

(c) A composition of linear fractional functions being linear fractional, $g_\theta$ is itself linear fractional, which is precisely the definition of a linear fractional Galton--Watson process (see \cite[Ch.~I.1]{AN72}).

The function $g_\theta$ has the two distinct fixed points $q(\theta)$ and $1$. Let us prove that $g_\theta'(q(\theta))g_\theta'(1)=1$. To this aim, we write $g_\theta(s)=(as+b)/(cs+d)$ with $\Delta:=ad-bc\neq0$, so that
$g_\theta'(s)=\Delta/(cs+d)^{2}$, and $c\neq0$. The fixed point equation reads $cs^{2}+(d-a)s-b=0$, so that its roots
$s_1:=q(\theta)$ and $s_2:=1$ satisfy $s_1s_2=-b/c$ and $s_1+s_2=(a-d)/c$.
Therefore 
$$(cs_1+d)(cs_2+d)=c^{2}s_1s_2+cd(s_1+s_2)+d^{2}=-bc+d(a-d)+d^{2}=\Delta,$$
and consequently 
$$g_\theta'(q(\theta))g_\theta'(1)=\Delta^{2}/\Delta^{2}=1.$$
As $$g_\theta'(1)=\E[\mathscr{P}^\theta(1)]=e^{\int_0^T\rho(\theta+u)du}=e^{Tf}$$ 
by \eqref{eq:phase-inv}, this proves \eqref{eq:multiplier}.

For \eqref{eq:LF-rates}, write $g:=g_\theta$, $q:=q(\theta)$ and $q_n:=g_{(n)}(0)=\P(\mathscr{P}^\theta(n)=0)$ for the probability of extinction before generation $n$, so that $q_n\uparrow q$. As $g$ is convex and nondecreasing, 
$$q-q_{n+1}=g(q)-g(q_n)\le g'(q)(q-q_n),$$ 
hence
$$q-q_n\le q\,e^{-nTf}\le e^{-nTf}$$ 
by \eqref{eq:multiplier}. For the second bound, the law of $\mathscr{P}^\theta(n)$ conditioned on being positive is
geometric with some parameter $1-\mathfrak b_n$, by the first paragraph of this proof applied at time $nT$, and taking expectations gives $(1-q_n)/(1-\mathfrak b_n)=\E[\mathscr{P}^\theta(n)]=e^{nTf}$. Therefore $1-\mathfrak b_n\le e^{-nTf}$ and
$$\P(0<\mathscr{P}^\theta(n)<A)=(1-q_n)\bigl(1-\mathfrak b_n^{\lceil A\rceil-1}\bigr)
\le A(1-\mathfrak b_n).$$
\end{proof}

\begin{proof}[Proof of Proposition \ref{prop_proba_inv}]
We apply Lemma \ref{lem:kendall} to $\beta=b_D(\bar{\mathbf n}(\cdot),0)$ and $\delta=d_D(\bar{\mathbf n}(\cdot),0)$, which are continuous and $T$-periodic, and note that $f=f(\bar{\mathbf n})$ by \eqref{def_inv_fitness}.
\end{proof}

To conclude this proof section, we will study the time needed to reach $0$ or $\eps K$. We will first focus on the subcritical case $f(\bar{\mathbf n})<0$, for which a reasoning in expectation is enough to obtain an upper bound, and then on the supercritical case $f(\bar{\mathbf n})>0$, which is more delicate to handle, as we are looking for an equivalent.

Before distinguishing between sub- and supercritical cases, notice that if we take $N_b(ds, d\theta)$ and $N_d(ds, d\theta)$ two independent Poisson random measures on $\R_+^2$ with intensity $ds d\theta$ we may realize the mutant population size process $N^K_D$ as the unique strong solution to
\begin{equation} \label{eq_ND} N_D^K(t) = N_D^K(0)+ \int_0^{t}\int_{\R_+} \mathbf{1}_{\{\theta \leq b_D(\mathbf Z^K(s-))N_D^K(s-)\}} N_b(ds, d\theta)- \int_0^{t}\int_{\R_+} \mathbf{1}_{\{\theta \leq d_D(\mathbf Z^K(s-))N_D^K(s-)\}} N_d(ds, d\theta). \end{equation}

\begin{proof}[Proof of \eqref{extinction}]
Assume $f(\bar{\mathbf n})<0$ and let $\eps$ be small enough that
$f(\bar{\mathbf n})+C_*\eps/T\le f(\bar{\mathbf n})/2<0$, with $C_*$ as in Proposition \ref{prop:env}.
From \eqref{eq_ND}, the process
\begin{equation}\label{eq:mart-M}
  M_D^K(t):=N_D^K(t)\exp\Big(-\int_0^t\rho_D\big(\mathbf{Z}^K(s)\big)ds\Big),
  \qquad t\geq 0,
\end{equation}
is a martingale for the natural filtration of the whole process $(\mathbf{N}^K,t\geq 0)$.

Let $c:=2T\sup\{|\rho_D(z)|:z\in\overline{\mathcal B_{\alpha_0}(\bar{\mathbf n})}\times[0,\eta_0]\}+C_*\eps$. Summing \eqref{eq:env-iii} over the complete periods and bounding the contribution of the last incomplete one, we get, on $\mathcal E_K$,
\begin{equation}\label{eq:int-rho}
  \Big|\int_0^t\rho_D(\mathbf Z^K(s))\,ds-tf(\bar{\mathbf n})\Big|
  \le c+\frac{C_*\eps}{T}\,t,\qquad t\le\tau_K .
\end{equation}
Since $f(\bar{\mathbf n})+C_*\eps/T\le f(\bar{\mathbf n})/2<0$, this gives $\int_0^t\rho_D(\mathbf Z^K(s))\,ds\le c+tf(\bar{\mathbf n})/2\le c$ for $t\le\tau_K$ on $\mathcal E_K$. The process $M^K_D$ of \eqref{eq:mart-M}, stopped at $\tau_K$, is a non-negative local martingale with $M^K_D(0)=1$, and a true martingale since $N^K_D(\cdot\wedge\tau_K)$ is dominated by a Yule process with rate $\bar\beta$. On $\mathcal E_K$, $N^K_D(t)\le e^cM^K_D(t)$ for $t\le\tau_K$. Hence, by Doob's maximal inequality,
\[
  \P\big(T^K_\eps\le\ln\ln K,\ \mathcal E_K\big)
  \le\P\Big(\sup_{t\le\tau_K}M^K_D(t)\ge e^{-c}\lfloor\eps K\rfloor\Big)
  \le\frac{e^c}{\lfloor\eps K\rfloor}.
\]
Moreover, with $t_K:=\ln\ln K<C\ln K$,
\[
  \P\big(N^K_D(t_K)\ge1,\ T^K_\eps>t_K,\ \mathcal E_K\big)
  \le e^{c+t_Kf(\bar{\mathbf n})/2}\,\E\big[M^K_D(t_K\wedge\tau_K)\big]
  =e^c(\ln K)^{f(\bar{\mathbf n})/2}.
\]
Therefore
\[
  \P\big(T^K_0>\ln\ln K\big)
  \le\frac{e^c}{\lfloor\eps K\rfloor}+e^c(\ln K)^{f(\bar{\mathbf n})/2}
  +\P(\mathcal E_K^c)\xrightarrow[K\to\infty]{}0
\]
by \eqref{eq:proba-E}. None of these estimates depends on the initial condition $x\in\mathcal B_\eps(\bar{\mathbf n})$, which proves \eqref{extinction}.
\end{proof}

\begin{proof}[Proof of \eqref{not_survive_ND} and \eqref{survive_ND}]
We now assume $f(\bar{\mathbf n})>0$. Throughout this step we apply Proposition \ref{prop:env} with the constant
$C':=2/f(\bar{\mathbf n})$, and we write $\tau_K:=T_\eps^K\wedge C'\ln K$; the value of $C'$ is chosen so that the environment is controlled up to the time at which the mutant population reaches the level $\eps K$, which is of order $\ln K/f(\bar{\mathbf n})$. We write $\theta_0:=\Theta(\Zperp(0))$ for the initial phase and $p_D(\theta_0)$ for the extinction probability given by \eqref{eq:kendall-sec3}. Two parameters have to be chosen: the number $n_1$ of periods over which the mutant population is compared with a birth and death process in a $T$-periodic environment, and the level $A$ that this population has to reach in order to be considered as established. Both are functions of $\eps$, fixed in \eqref{eq:choice-n1-A} below, and the limit $K\to\infty$ is taken last. The reason why $n_1$ may be allowed to grow as $\eps$ decreases - which is what makes the final estimate quantitative - is Step (4b): the comparison there bears on the generating functions of the processes over one period, and the iteration is then controlled by the contraction rate \eqref{eq:multiplier}, uniformly in the number of periods.

\medskip\noindent\textbf{(4a) Coupling over a bounded number of periods.}
We fix once and for all
\begin{equation}\label{eq:choice-n1-A}
   n_1:=\Bigl\lceil\frac{\ln(1/\eps)}{T f(\bar{\mathbf n})}\Bigr\rceil ,
   \qquad A:=\bigl\lceil\eps^{-1/2}\bigr\rceil ,
   \qquad \gamma:=L_{b,d}\big(1+C_*+L_{\bar{\mathbf n}}C_*n_1\big)\eps ,
\end{equation}
where $L_{b,d}$ is a common Lipschitz constant of $b_D$ and $d_D$ on $\overline{\mathcal B_{\alpha_0}(\bar{\mathbf n})}\times[0,\eta_0]$ and $L_{\bar{\mathbf n}}:=\sup_t\|\mathbf r(\bar{\mathbf n}(t))\|$ is the speed along the cycle, and we note that $\gamma\le C\eps\ln(1/\eps)$ for $\eps$ small enough. By \eqref{eq:env-i}--\eqref{eq:env-ii}, on $\mathcal E_K$ and for $t\in[nT,(n+1)T]$ with $(n+1)T\le n_1T\wedge\tau_K$, $\|\Zperp(t)-\bar{\mathbf n}(t+\theta_0)\|\le C_*\eps+L_{\bar{\mathbf n}}|\theta_n-\theta_0|\le C_*\eps+L_{\bar{\mathbf n}}nC_*\eps$, and $Z^K_D(t)\le\eps$. Hence, on $\mathcal E_K$ and for $t\le n_1T\wedge\tau_K$,
$$
  \big|b_D(\mathbf{Z}^K(t))-b_D(\bar{\mathbf n}_{\theta_0}(t),0)\big|\vee
  \big|d_D(\mathbf{Z}^K(t))-d_D(\bar{\mathbf n}_{\theta_0}(t),0)\big|\leq \gamma .
$$
Define the $T$-periodic rates
$$b^{\pm}:=\big(b_D(\bar{\mathbf n}_{\theta_0}(\cdot),0)\pm\gamma\big)^{+} \quad \text{and} \quad d^{\mp}:=\big(d_D(\bar{\mathbf n}_{\theta_0}(\cdot),0)\mp\gamma\big)^{+}, $$
and let $P^{\pm}$ be the corresponding linear birth-and-death processes, constructed from the same Poisson random measures $N_b,N_d$ as $N_D^K$ in \eqref{eq_ND}. Then, on $\mathcal E_K$,
\begin{equation}\label{eq:coupling}
  P^{-}(t)\ \le\ N_D^K(t)\ \le\ P^{+}(t),
  \qquad t\le n_1T\wedge\tau_K .
\end{equation}
Both $P^{\pm}$ have $T$-periodic rates, with invasion fitnesses $f^\pm$ satisfying $|f^\pm-f(\bar{\mathbf n})|\le2\gamma$, so that Lemma \ref{lem:kendall} applies to them.

\medskip\noindent\textbf{(4b) Comparison of the extinction probabilities.}
Let $g^{\pm}$ denote the generating functions of $P^{\pm}(T)$ and let $q^{\pm}$ be their fixed points in $[0,1)$. The process $P_D^{\theta_0}$ can be constructed from the same Poisson random measures $N_b$ and $N_d$ as $P^{\pm}$, and its rates satisfy $b^-\le b_D(\bar{\mathbf n}_{\theta_0}(\cdot),0)\le b^+$ and $d^-\le d_D(\bar{\mathbf n}_{\theta_0}(\cdot),0)\le d^+$. Hence $P^-(t)\le P_D^{\theta_0}(t)\le P^+(t)$ for every $t\ge0$, and therefore $q^{+}\le p_D(\theta_0)\le q^{-}$. The key point is that $g^+$ and $g^-$ are compared over one period only: since $P^{-}(T)\le P^{+}(T)$ almost surely,
\begin{equation}\label{eq:gpm-close}
  \big\Vert g^{+}-g^{-}\big\Vert_{\infty,[0,1]}
   \leq \P\big(P^{+}(T)\neq P^{-}(T)\big)
   \leq \E\big[P^{+}(T)-P^{-}(T)\big] \leq C_\gamma\gamma ,
\end{equation}
where $C_\gamma$ depends only on $T$, on $f(\bar{\mathbf n})$ and on the bounds of the rates along the cycle, but not on $n_1$. Write $q_n^{\pm}:=g^{\pm}_{(n)}(0)$ for the probabilities that $P^{\pm}$ is extinct at time $nT$. Since $P^-\le P^+$, we have $q_n^{+}\le q_n^{-}\le q^{-}$ for every $n$. By \eqref{eq:multiplier} applied to $P^-$, and since $g^{-}$ is convex and nondecreasing,
$$
  \bar\alpha:=\sup_{s\le q^{-}}(g^{-})'(s)=(g^{-})'(q^{-})=e^{-Tf^{-}}
   \leq e^{-T(f(\bar{\mathbf n})-2\gamma)}< 1
$$
for $\eps$ small enough. Writing $q_{n+1}^{+}-q_{n+1}^{-}=\big(g^{+}(q_n^{+})-g^{-}(q_n^{+})\big)+\big(g^{-}(q_n^{+})-g^{-}(q_n^{-})\big)$ and using the mean value theorem on $[0,q^-]$, we get
$$|q_{n+1}^{+}-q_{n+1}^{-}|\le\Vert g^{+}-g^{-}\Vert_\infty
+\bar\alpha|q_n^{+}-q_n^{-}|. $$
Hence, since $q_0^{+}=q_0^{-}=0$, using \eqref{eq:gpm-close} we obtain
\begin{equation}\label{eq:un-vn}
  \sup_{n\in\N} |q_n^{+}-q_n^{-}| \leq \frac{C_\gamma\gamma}{1-\bar\alpha}
  \leq C_2\eps\ln(1/\eps) .
\end{equation}
Applying \eqref{eq:LF-rates} to $P^{-}$ and by the choice \eqref{eq:choice-n1-A} of $n_1$, we also have
$$q^{-}-q_{n_1}^{-}\le e^{-n_1Tf^{-}}\le C_3\eps. $$
By \eqref{eq:coupling}, $P^-(n_1T)\le N_D^K(n_1T)\le P^+(n_1T)$ on $\mathcal E_K\cap\{n_1T\le\tau_K\}$. For $K$ large enough, $C'\ln K>n_1T$, so that $\{n_1T>\tau_K\}\subset\{T^K_\eps\le n_1T\}$. Therefore
$$
  q_{n_1}^{+}-\P(\mathcal E_K^c)-\P(T^K_\eps\le n_1T)
  \ \le\ \P\big(N_D^K(n_1T)=0\big)
  \ \le\ q_{n_1}^{-}+\P(\mathcal E_K^c)+\P(T^K_\eps\le n_1T).
$$
Since moreover $q_{n_1}^{+}\le q^{+}\le p_D(\theta_0)\le q^{-}$ and $q_{n_1}^{-}\le q^{-}$, the two quantities $\P(N_D^K(n_1T)=0)$ and $p_D(\theta_0)$ differ by at most
$$
  q^{-}-q_{n_1}^{+}+\P(\mathcal E_K^c)+\P(T^K_\eps\le n_1T)
  \le\big(q^{-}-q_{n_1}^{-}\big)+\big(q_{n_1}^{-}-q_{n_1}^{+}\big)+\P(\mathcal E_K^c)+\P(T^K_\eps\le n_1T).
$$
We conclude, by \eqref{eq:un-vn} and the previous bound on $q^{-}-q_{n_1}^{-}$, that
\begin{align}\label{eq:step4b-1}
  \big|\P\big(N_D^K(n_1T)=0\big)-p_D(\theta_0)\big| &  \leq C_4\eps\ln(1/\eps)+\P(\mathcal E_K^c)+\P(T^K_\eps\le n_1T).
\end{align}

\medskip\noindent\textbf{(4c) The mutant population is either extinct or large at time $n_1T$.} Using \eqref{eq:coupling}, then \eqref{eq:un-vn} and \eqref{eq:LF-rates} applied to $P^{-}$, we obtain
\begin{align}
  \P\big(0<N_D^K(n_1T)<A\big)
  &\leq \Big[\P\big(P^{+}(n_1T)\geq 1\big)-\P\big(P^{-}(n_1T)\geq 1\big)\Big]
    +\P\big(0<P^{-}(n_1T)<A\big)+\P(\mathcal E_K^c)\notag\\
  &\leq \frac{C_\gamma \gamma}{1-\bar\alpha}+A\,e^{-n_1Tf^{-}}+\P(\mathcal E_K^c)
    \leq C_5\sqrt{\eps}+\P(\mathcal E_K^c),\label{eq:step4b-2}
\end{align}
the second term being of order $\eps^{-1/2}\cdot\eps=\sqrt\eps$ by \eqref{eq:choice-n1-A}.

\medskip\noindent\textbf{(4d) Growth phase: upper bound on $T_\eps^K$.}
The event $\mathcal E_K$ is not adapted to the filtration, so we first localise. Let
\[
  \varrho_K:=\inf\Big\{t\ge0:\ \big\|\Zperp(t)
  -n^{\Zperp(\lfloor t/T\rfloor T)}\big(t-\lfloor t/T\rfloor T\big)\big\|
  >C_0\eps\Big\}.
\]
This is a stopping time. The induction in the proof of Proposition~\ref{prop:env} shows that \eqref{eq:env-i}--\eqref{eq:env-iii}, hence \eqref{eq:int-rho}, hold pathwise for $t\le\varrho_K\wedge\tau_K$, and that $\{\varrho_K<\tau_K\}\subset\mathcal E_K^c$. In this step, the martingale $\widehat M$ below is stopped at $\varrho_K\wedge\tau_K$, and the event $\mathcal E_K$ is replaced by $\{\varrho_K\ge\tau_K\}$, whose complement has probability at most $\P(\mathcal E_K^c)$.
On the event $\{N_D^K(n_1T)\geq A\}\cap\mathcal E_K$, consider
$$  \widehat M(t):=N_D^K(n_1T+t)\exp\Big(-\int_{n_1T}^{n_1T+t}
     \rho_D(\mathbf{Z}^K(s))ds\Big),\qquad t\leq \tau_K-n_1T ,$$
which is again a martingale, with $\widehat M(0)=N_D^K(n_1T)\ge A$. Its predictable bracket is
$$  \langle \widehat M\rangle_t=\int_0^{t}
   e^{-2\int_{n_1T}^{n_1T+s}\rho_D(\mathbf{Z}^K)}\,N_D^K(n_1T+s)\,
   \sigma_D\bigl(\mathbf{Z}^K(n_1T+s)\bigr)ds ,$$
hence, using \eqref{eq:env-iii} and $\sigma_D\le\bar\sigma$ on the relevant compact set,
$$  \E\bigl[\langle\widehat M\rangle_{\tau_K-n_1T}\mid\mathcal F_{n_1T}\bigr]
  \le \bar\sigma\,\widehat M(0)\int_0^{\infty}e^{-s(f(\bar{\mathbf n})-C_*\eps/T)+c}ds
  = \frac{\bar\sigma e^{c}}{f(\bar{\mathbf n})-C_*\eps/T}\,\widehat M(0).$$
Doob's $L^2$ inequality then gives
\begin{equation}\label{eq:kappaA}
  \P\Big(\sup_{t}\big|\widehat M(t)-\widehat M(0)\big|
     \geq \tfrac12\widehat M(0) \Big|\ \mathcal F_{n_1T}\Big)
   \leq \frac{4\bar\sigma e^{c}}{(f(\bar{\mathbf n})-C_*\eps/T)A}
   \leq C_6\sqrt{\eps} .
\end{equation}
On the complementary event, for $t\le\tau_K-n_1T$,
$$  N_D^K(n_1T+t) \geq \tfrac{A}{2}e^{-c}
      \exp\big\{t\big(f(\bar{\mathbf n})-C_*\eps/T\big)\big\},$$
which is larger than $\eps K$ as soon as $t\geq \ln K/(f(\bar{\mathbf n})-C_*\eps/T)\cdot(1+o_K(1))$. As $C'\geq 2/f(\bar{\mathbf n})$, the truncation at $C'\ln K$ in $\tau_K$ is not active for $\eps$ small enough, and
we conclude
\begin{equation}\label{eq:Teps-upper}
  \limsup_{K\to\infty}\ \P\Big(N_D^K(n_1T)\geq A,
   T_\eps^K>\frac{\ln K}{f(\bar{\mathbf n})-C_*\eps/T}\Big) \leq C_6\sqrt{\eps} .
\end{equation}

\medskip\noindent\textbf{(4e) Growth phase: lower bound on $T_\eps^K$.}
By \eqref{eq:mart-M}, optional stopping and \eqref{eq:env-iii}, for $t\le C'\ln K$,
$$
  \E\big[N_D^K(t\wedge\tau_K)\mathbf{1}_{\mathcal E_K}\big]
   \leq e^{c}e^{t(f(\bar{\mathbf n})+C_*\eps/T)} ,
$$
so that Markov's inequality gives
$$
  \P\bigl(T_\eps^K\le t\bigr)\le
  \frac{e^{c}e^{t(f(\bar{\mathbf n})+C_*\eps/T)}}{\eps K}+\P(\mathcal E_K^{c}).
$$
Choosing $t=\ln K/\bigl(f(\bar{\mathbf n})+2C_*\eps/T\bigr)$ makes the first term equal to 
$$e^{c}\eps^{-1}K^{\frac{f(\bar{\mathbf n})+C_*\eps/T}{f(\bar{\mathbf n})+2C_*\eps/T}-1}, $$
which tends to $0$ as $K\to\infty$. Hence
\begin{equation}\label{eq:Teps-lower}
  \lim_{K\to\infty}\P\Big(T_\eps^K\leq\frac{\ln K}{f(\bar{\mathbf n})+2C_*\eps/T}\Big)=0 .
\end{equation}

\medskip\noindent\textbf{(4f) Conclusion.}
It remains to relate the event $\{T_0^K<T_\eps^K\}$ appearing in \eqref{not_survive_ND} to the event $\{N_D^K(n_1T)=0\}$ estimated in \eqref{eq:step4b-1}. These two events differ by two others only. The first is $\{T_\eps^K\leq n_1T\}$, whose probability is $o_K(1)$: by the first moment bound of Step (4e) with $t=n_1T$, which is a fixed time, 
$$\P(T_\eps^K\leq n_1T)\leq e^{c}e^{n_1T(f(\bar{\mathbf n})+C_*\eps/T)}/(\eps K)+\P(\mathcal E^c_K), $$ 
and $n_1$ does not depend on $K$. The second is the event that the mutant population is still alive at time $n_1T$ but becomes extinct before reaching the level $\eps K$; on $\mathcal E_K$ this implies either $0<N_D^K(n_1T)<A$, or the failure of the Doob estimate \eqref{eq:kappaA}, since on the complement of both the lower bound of Step (4d) gives $N_D^K(t)\ge1$ for every
$t\in[n_1T,\tau_K]$, and $0$ being absorbing this in turn forces $N_D^K(t)\geq 1$ for every $t\leq\tau_K$. Its probability is therefore at most
$$C_5\sqrt\eps+C_6\sqrt\eps+\P(\mathcal E_K^c)$$ 
by \eqref{eq:step4b-2} and \eqref{eq:kappaA}.

Adding \eqref{eq:step4b-1}, \eqref{eq:step4b-2}, \eqref{eq:kappaA},
\eqref{eq:Teps-upper} and \eqref{eq:Teps-lower}, recalling that $\P(\mathcal E_K^c)\to0$ as $K\to\infty$ by \eqref{eq:proba-E}, and setting $\tilde{C}:= C_* (1 \vee 2/T)$ so that $2C_*\eps/T\leq \tilde{C} \eps$, we obtain the existence of a finite constant $C_{**}$, depending only on $T$, on $f(\bar{\mathbf n})$ and on the bounds of the rates along the cycle, such that
$$
  \limsup_{K\to\infty}\Big|\P\big(T_0^K<T_\eps^K\big)-p_D(\theta_0)\Big|
    \leq C_{**}\sqrt{\eps}
$$
and
$$
  \limsup_{K\to\infty}\left|\P\Bigl(
    \frac{\ln K}{f(\bar{\mathbf n})+\tilde{C}\eps}<T_\eps^K<\frac{\ln K}{f(\bar{\mathbf n})-\tilde{C}\eps}
   \Bigr)-\bigl(1-p_D(\theta_0)\bigr)\right|\ \le\ C_{**}\sqrt{\eps},
$$
which are \eqref{not_survive_ND} and \eqref{survive_ND}.
\end{proof}

\begin{proof}[Proof of Corollary \ref{cor:average}]
By \eqref{eq:kendall-sec3} and the periodicity of $b$ and $\rho$, $u:=1/s$ is given by
\[
u(\theta)=\frac{1}{1-e^{-Tf(\bar{\mathbf n})}}
\int_\theta^{\theta+T}b(r)\,e^{-\int_\theta^r\rho(v)\,dv}\,dr .
\]
Differentiating, and using $\int_\theta^{\theta+T}\rho=Tf(\bar{\mathbf n})$ and $b(\theta+T)=b(\theta)$, we get $u'=\rho u-b$, which is
\eqref{eq:riccati}. Since $\int_0^T\rho\neq0$, this linear equation has a unique $T$-periodic solution. Dividing \eqref{eq:riccati} by $s>0$ and integrating over one period gives $0=\int_0^T(bs-\rho)$, that is \eqref{eq:average}. The last formula follows from
$\frac1T\int_0^T bs=\bar b\,\bar s+\operatorname{Cov}(b,s)$.
\end{proof}

\section*{Statements and Declarations}

\subsection*{Funding}

S. Billiard, C. Mian, and C. Smadi benefited from the support of the MMB Chair (Mathematics for Life and Environmental Sciences) VEOLIA-Ecole Polytechnique-MNHN-F.X.

Funding acknowledgements by A.~Tóbiás: This paper was supported by the János Bolyai Research Scholarship of the Hungarian Academy of Sciences. Project no. STARTING 149835 has been implemented with the support provided by the Ministry of Culture and Innovation of Hungary from the National Research, Development and Innovation Fund, financed under the STARTING 24 funding scheme.

\subsection*{Competing interests}

The authors have no relevant financial or non-financial interests to disclose.

\subsection*{Data availability}

No external datasets were used in this study. All numerical results were generated from simulations of the mathematical models described in the manuscript.

\bibliography{periodic}

@article{Abrams2000,
author={Abrams, P.A.},
title={The Evolution of Predator-Prey Interactions: Theory and Evidence},
journal={Annual Review of Ecology and Systematics},
volume={31},
pages={79--105},
year={2000}
}

@book{AN72,
title= "Branching {P}rocesses",
publisher="Springer",
author="K.B. Athreya and P.E. Ney",
year="1972"
}

@article{cortez2020,
  title={Destabilizing evolutionary and eco-evolutionary feedbacks drive empirical eco-evolutionary cycles},
  author={Cortez, M.H. and Patel, S. and Schreiber, S.J.},
  journal={Proceedings of the Royal Society B: Biological Sciences},
  volume={287},
  number={1919},
  year={2020},
  publisher={The Royal Society}
}

@article{loeuille2005,
  title={Evolutionary emergence of size-structured food webs},
  author={Loeuille, N. and Loreau, M.},
  journal={Proceedings of the National Academy of Sciences},
  volume={102},
  number={16},
  pages={5761--5766},
  year={2005},
  publisher={National Academy of Sciences}
}

@article{hubbard2010,
  title={Vertebrate pigmentation: from underlying genes to adaptive function},
  author={Hubbard, J.K. and Uy, J.A.C. and Hauber, M.E. and Hoekstra, H.E. and Safran, R.J.},
  journal={Trends in Genetics},
  volume={26},
  number={5},
  pages={231--239},
  year={2010},
  publisher={Elsevier}
}

@book{tilman1982,
  title={Resource Competition and Community Structure},
  author={Tilman, D.},
  series={Monographs in Population Biology},
  number={17},
  year={1982},
  publisher={Princeton University Press}
}

@article{jones2007,
  title={Effects of rapid prey evolution on predator--prey cycles},
  author={Jones, L.E. and Ellner, S.P.},
  journal={Journal of Mathematical Biology},
  volume={55},
  number={4},
  pages={541--573},
  year={2007},
  publisher={Springer}
}

@article{abrams1999,
  title={Is predator-mediated coexistence possible in unstable systems?},
  author={Abrams, P.A.},
  journal={Ecology},
  volume={80},
  number={2},
  pages={608--621},
  year={1999},
  publisher={Wiley Online Library}
}

@article{durrett2010,
  title={Evolution in predator--prey systems},
  author={Durrett, R. and Mayberry, J.},
  journal={Stochastic Processes and their Applications},
  volume={120},
  number={7},
  pages={1364--1392},
  year={2010},
  publisher={Elsevier}
}

@Article{billiard2017interplay,
  Title                    = {The interplay of two mutations in a population of varying size: a stochastic eco-evolutionary model for clonal interference},
  Author                   = {Billiard, S. and Smadi, C.},
  Journal                  = {Stochastic Processes and their Applications},
  Year                     = {2017},
  Number                   = {3},
  Pages                    = {701--748},
  Volume                   = {127},
  Publisher                = {Elsevier}
}

@article{BS19,
AUTHOR={Billiard, S. and Smadi, C.},
TITLE={Stochastic dynamics of three competing clones: Conditions and times for invasion, coexistence, and fixation},
JOURNAL={The American Naturalist},
VOLUME={195},
NUMBER={3},
PAGES={463--484},
YEAR={2020},
}

@article{BCFMT16,
author="S. Billiard and P.  Collet and R. Ferrière and S. Méléard and V. Tran",
title="The effect of competition and horizontal trait inheritance on invasion, fixation, and polymorphism",
journal="Journal of Theoretical Biology",
volume="411",
pages="48-58",
year="2016"}

@incollection{BCFMT18,
  author    = {Billiard, Sylvain and Collet, Pierre and Ferri{\`e}re, R{\'e}gis and M{\'e}l{\'e}ard, Sylvie and Tran, Viet Chi},
  title     = {Stochastic dynamics for adaptation and evolution of microorganisms},
  booktitle = {European {C}ongress of {M}athematics ({B}erlin, {J}uly 18--22, 2016)},
  editor    = {Mehrmann, Volker and Skutella, Martin},
  publisher = {European Mathematical Society},
  address   = {Z{\"u}rich},
  pages     = {525--550},
  year      = {2018},
  isbn      = {978-3-03719-176-7},
  doi       = {10.4171/176-1/25}
}

@article{bovier2019crossing,
  title={Crossing a fitness valley as a metastable transition in a stochastic population model},
  author={Bovier, A. and Coquille, L. and Smadi, C.},
  journal={Annals of Applied Probability},
  volume={29},
  number={6},
  pages={3541--3589},
  year={2019},
  publisher={JSTOR}
}

@article{MBLS,
title={Dynamics of two species with density-dependent
interactions in a mutualistic context},
author={C. Mian and S. Billiard and V. Llaurens and C. Smadi},
journal={arXiv:2509.06062},
year={2026}
}

@article{champagnat2006microscopic,
  title={A microscopic interpretation for adaptive dynamics trait substitution sequence models},
  author={Champagnat, N.},
  journal={Stochastic Processes and their Applications},
  volume={116},
  number={8},
  pages={1127--1160},
  year={2006},
  publisher={Elsevier}
}

@article{champagnat2007invasion,
  title={Invasion and adaptive evolution for individual-based spatially structured populations},
  author={Champagnat, N. and M{\'e}l{\'e}ard, S.},
  journal={Journal of Mathematical Biology},
  volume={55},
  pages={147--188},
  year={2007},
  publisher={Springer}
}

@Article{champagnat2011polymorphic,
  Title                    = {Polymorphic evolution sequence and evolutionary branching},
  Author                   = {Champagnat, N. and M{\'e}l{\'e}ard, S.},
  Journal                  = {Probability Theory and Related Fields},
  Year                     = {2011},
  Number                   = {1-2},
  Pages                    = {45--94},
  Volume                   = {151},
  Publisher                = {Springer}
}

@article{champagnat2021stochastic,
  title={Stochastic analysis of emergence of evolutionary cyclic behavior in population dynamics with transfer},
  author={Champagnat, N. and M{\'e}l{\'e}ard, S. and Tran, V.C.},
  journal={The Annals of Applied Probability},
  volume={31},
  number={4},
  pages={1820--1867},
  year={2021},
  publisher={JSTOR}
}

@article{CoquilleKrautSmadi,
title={Stochastic individual-based models with power law mutation
rate on a general finite trait space},
author={L. Coquille and A. Kraut and C. Smadi},
journal={Electronic Journal of Probability},
volume={26},
pages={1--37},
year={2021}
}

@article{DieckmannDoebeli,
title={On the origin of species by sympatric speciation},
author={U. Dieckmann and M. Doebeli},
journal={Nature},
volume={400},
pages={354--357},
year={1999}
}

@book{dupuis1997weak,
  title={A Weak Convergence Approach to the Theory of Large Deviations},
  author={Dupuis, P. and Ellis, R.S.},
  year={2011},
  publisher={John Wiley \& Sons}
}

@book{ethier2009markov,
  title={Markov Processes: Characterization and Convergence},
  author={Ethier, S.N. and Kurtz, T.G.},
  year={2009},
  publisher={John Wiley \& Sons}
}

@article{F02,
author={R. Ferrière and J.L. Bronstein and S. Rinaldi and R. Law and M. Gauduchon},
title={Cheating and the evolutionary stability of mutualisms},
journal={Proceedings
of the Royal Society B: Biological Sciences},
volume={269},
number={1493},
pages={773--780},
year={2002}
}

@book{FW84,
title={Random Perturbations of Dynamical Systems},
author={Freidlin, M. and Wentzell, A.D.}, 
series={Grundlehren der Mathematischen Wissenschaften (Fundamental Principles of Mathematical Sciences)},
volume={260},
publisher={Springer},
year={1984}
}

@book{kuznetsov,
title={Elements of Applied
Bifurcation Theory,
Second Edition},
author={Y.A. Kuznetsov},
publisher={Springer},
series={Applied Mathematical Sciences},
volume={112},
year={1998}
}

@article{RosenzweigMacArthur,
author={M.L. Rosenzweig and
R.H. MacArthur}, 
title={Graphical representation and stability conditions of predator-prey interaction},
journal={The American Naturalist},
volume={97},
pages={209–223},
year={1963}
}

@article{Smale1976,
author={S. Smale},
title={On the differential equations of species in competition},
journal={Journal of Mathematical Biology}, 
volume={3},
pages={5--7},
year={1976}
}

@book{Turchin2003,
title={Complex Population Dynamics: A Theoretical/Empirical Synthesis},
author={P. Turchin},
publisher={Princeton University Press},
city={Princeton},
year={2003}
}

@article{Vano,
title={Chaos in low-dimensional {L}otka–{V}olterra models of
competition},
author={
J.A. Vano and J.C. Wildenberg and M.B. Anderson and J.K. Noel and J.C. Sprott},
journal={Nonlinearity},
volume={19},
pages={2391--2404},
year={2006}}

@article{WX,
author={R. Wang and D. Xiao},
title={Bifurcations and chaotic dynamics in a 4-dimensional
competitive {L}otka--{V}olterra system},
journal={Nonlinear Dynamics},
volume={59},
pages={411--422},
year={2010}
}

@article{gillespie1977exact,
  title={Exact stochastic simulation of coupled chemical reactions},
  author={Gillespie, D.T.},
  journal={The Journal of Physical Chemistry},
  volume={81},
  number={25},
  pages={2340--2361},
  year={1977},
  publisher={ACS Publications}
}

@article{Kendall1948, 
author={Kendall, D. G.},
title={On the generalized ``birth-and-death'' process},
journal={Ann. Math. Statist.}, 
volume={19}, 
pages={1--15}, 
year={1948}
}

@book{Hartman, 
author={Hartman, P.},
title={Ordinary Differential Equations}, 
edition={2}, 
publisher={SIAM},
year={2002}
}

@article{Bennett1962, 
author={Bennett, G.},
title={Probability inequalities for the sum of independent random variables}, journal={J. Amer. Statist. Assoc.}, 
volume={57},
pages={33--45}, 
year={1962}
}

@book{BLM2013, 
author={Boucheron, S. and Lugosi, G. and Massart, P.},
title={Concentration Inequalities: A Nonasymptotic Theory of Independence}, publisher={Oxford University Press}, 
year={2013}
}

@article{Guckenheimer75, 
author={Guckenheimer, J.},
title={Isochrons and phaseless sets}, 
journal={Journal of Mathematical Biology},
volume={1}, 
pages={259--273}, 
year={1975}
}

@book{Chicone, 
author={Chicone, C.},
title={Ordinary Differential Equations with Applications}, 
edition={2},
publisher={Springer}, 
year={2006}
}

@article{MetzNisbetGeritz1992,
  author  = {Metz, J.A.J. and Nisbet, R.M. and Geritz, S.A.H.},
  title   = {How should we define `fitness' for general ecological scenarios?},
  journal = {Trends in Ecology \& Evolution},
  volume  = {7}, number = {6}, pages = {198--202}, year = {1992}}

@article{FerriereGatto1995,
  author  = {Ferri\`ere, R. and Gatto, M.},
  title   = {Lyapunov exponents and the mathematics of invasion in oscillatory or chaotic populations},
  journal = {Theoretical Population Biology},
  volume  = {48}, number = {2}, pages = {126--171}, year = {1995},
  doi     = {10.1006/tpbi.1995.1024}}

@article{UeckerHermisson2011,
  author  = {Uecker, H. and Hermisson, J.},
  title   = {On the fixation process of a beneficial mutation in a variable environment},
  journal = {Genetics},
  volume  = {188}, number = {4}, pages = {915--930}, year = {2011},
  doi     = {10.1534/genetics.110.124297}}

@article{CarmonaGandon2020,
  author  = {Carmona, P. and Gandon, S.},
  title   = {Winter is coming: pathogen emergence in seasonal environments},
  journal = {PLoS Computational Biology},
  volume  = {16}, number = {7}, pages = {e1007954}, year = {2020}}

@article{EsserKraut2025,
  author  = {Esser, M. and Kraut, A.},
  title   = {Effective growth rates in a periodically changing environment: from mutation to invasion},
  journal = {Stochastic Processes and their Applications},
  volume  = {184}, pages = {104598}, year = {2025},
  doi     = {10.1016/j.spa.2025.104598}}

@article{Gilpin1975,
  author  = {Gilpin, M.E.},
  title   = {Limit cycles in competition communities},
  journal = {The American Naturalist},
  volume  = {109}, number = {965}, pages = {51--60}, year = {1975},
  doi     = {10.1086/282973}}

@article{SinervoLively1996,
  author  = {Sinervo, B. and Lively, C.M.},
  title   = {The rock--paper--scissors game and the evolution of alternative male strategies},
  journal = {Nature},
  volume  = {380}, pages = {240--243}, year = {1996}}

@article{BaireyKelsicKishony2016,
  author  = {Bairey, E. and Kelsic, E.D. and Kishony, R.},
  title   = {High-order species interactions shape ecosystem diversity},
  journal = {Nature Communications},
  volume  = {7}, pages = {12285}, year = {2016},
  doi     = {10.1038/ncomms12285}}
\bibliographystyle{plain}

\end{document}